\documentclass[12pt,reqno]{amsart}
\usepackage{amsmath,amssymb,bbm,subfig,url,ulem}%,showlabels}
\usepackage{amsthm}
\usepackage{xcolor}
\usepackage{listings}
\usepackage[T1]{fontenc}
\usepackage[margin=1in]{geometry}

\usepackage{tikz}
\usepackage[bookmarks=false,unicode,colorlinks,urlcolor=red,citecolor=red,linkcolor=blue]{hyperref}

\usepackage{subfig}
\usetikzlibrary{calc}
\usetikzlibrary{arrows}
\usetikzlibrary{patterns}
\usetikzlibrary{through}
\usetikzlibrary{plotmarks}

\numberwithin{equation}{section}

\usepackage{aliascnt}
\newtheorem{theorem}{Theorem}[section]

\newaliascnt{lemma}{theorem}
\newtheorem{lemma}[lemma]{Lemma}
\aliascntresetthe{lemma}

\newaliascnt{proposition}{theorem}

\aliascntresetthe{proposition}

\newaliascnt{corollary}{theorem}
\newtheorem{corollary}[corollary]{Corollary}
\aliascntresetthe{corollary}

\newaliascnt{conjecture}{theorem}

\aliascntresetthe{conjecture}

\newaliascnt{remark}{theorem}
\newtheorem{remark}[remark]{Remark}
\aliascntresetthe{remark}

\newaliascnt{definition}{theorem}
\newtheorem{definition}[definition]{Definition}
\aliascntresetthe{definition}

\makeatletter
\def\tagform@#1{\maketag@@@{\ignorespaces#1\unskip\@@italiccorr}}
\let\orgtheequation\theequation
\def\theequation{(\orgtheequation)}
\makeatother
\def\equationautorefname~{}

\newcommand{\Real}{{\mathbb{R}}}

\newcommand{\Id}{\mathbb{I}\mathrm{d}}
\newcommand{\Rd}{{\Real^d}}

\newcommand{\E}{{\bf E}}

\DeclareMathOperator{\tr}{Tr}

\newcommand{\mref}[1]{%
\href{http://www.ams.org/mathscinet-getitem?mr=#1}{#1}}

\newcommand{\arxiv}[1]{%
\href{http://front.math.ucdavis.edu/#1}{ArXiv:#1}}

 \newcommand{\doi}[1]{%
 \href{http://dx.doi.org/#1}{doi:#1}
 }

\begin{document}

\title[Generalized tight $P$-frames and spectral bounds for Laplace-like operators.]{Generalized tight $p$-frames and spectral bounds for Laplace-like operators.}

\author{B. A. Siudeja}
\address{Department of Mathematics, Univ.\ of Oregon, Eugene,
OR 97403, U.S.A.}
\email{Siudeja\@@uoregon.edu}

\date{\today}
%\keywords{chi-squared distribution, symmetric polynomials, Fourier multipliers, affine transformation
\begin{abstract}
  We prove sharp upper bounds for sums of eigenvalues (and other spectral functionals) of Laplace-like operators, including bi-Laplacian and fractional Laplacian. We show that among linear images of a highly symmetric domain, our spectral functionals are maximal on the original domain. We exploit the symmetries of the domain, and the operator, avoiding necessity of finding good test functions for variational problems. This is especially important for fractional Laplacian, since exact solutions are not even known on intervals, making it hard to find good test functions. 

To achieve our goals we generalize tight $p$-fusion frames, to extract the best possible geometric results for domains with isometry groups admitting tight $p$-frames. Any such group generates a tight $p$-fusion frame via conjugation of arbitrary projection matrix. We show that generalized tight $p$-frames can also be obtained by conjugation of arbitrary rectangular matrix, with frame constant depending on the singular values of the matrix.
\end{abstract}

\maketitle
\section{Introduction}

The dependence of the Laplace eigenvalues on the shape of the domain was extensively studied by many authors. In particular, shapes that maximize/minimize spectral functionals are of great interest (see monograph \cite{He06} for a comprehensive overview). We find extremizing domains for spectral functionals of a broad family of Laplace-like operators, ranging from bi-Laplacian (and higher order operators) to fractional Laplacian (and other non-local operators). The important common feature of all treated operators is invariance under isometries of the Euclidean space. We derive bounds for sums of eigenvalues of such operators on linearly transformed highly symmetric domains in terms of the eigenvalues on the symmetric domain. We obtain sharp bounds based purely on the symmetries of the domain and the operator, regardless if the eigenvalues and eigenfunctions of the symmetric domain are known explicitly. See \autoref{mainres} for bounds involving commonly used operators. This work extends \cite{LS11} and \cite{LS11plane} to more general operators. 

Our proofs are based on the theory of tight $p$-frames, recently studied by Bachoc, Ehler and Okoudjou (\cite{EO12,BE13}). In \autoref{pmatrix}, we generalize certain aspects of this theory, to find an exact method of evaluating quadratic forms on transformed domains. This part of the paper may be of independent interest. 

As discovered by Bachoc and Ehler \cite{BE13}, existence of tight $p$-frames is strongly connected to the theory of invariant polynomials for irreducible representations of finite groups of symmetries (subgroups of orthogonal groups). We emphasize and exploit this connection even further in \autoref{subspmatrix} to generalize $p$-fusion frames involving subspace projection matrices, to arbitrary matrices. 

The frame constants for such generalizations are particularly hard to evaluate. We meet the challenge (in \autoref{secfp}) using a mixture of techniques spanning variety of fields of mathematics. We use probabilistic arguments to establish two combinatorial formulas involving symmetric polynomials. We use algebra of symmetric functions on abstract alphabet (including alphabet manipulations) to relate our combinatorial identities. Finally, cycle index of the symmetric group provides the simplest form of the generalized $p$-frame constants. As a corollary of that result we also get a formula for the moments of a generalized $\chi^2$-distribution (sum of squares of centered Gaussian random variables with arbitrary variances), \autoref{chisquared}.

\subsection{Main results}\label{mainres}
In this section we present the eigenvalue estimates for a few commonly used Laplace-like operators. Most of these results are special cases of the general Fourier multiplier eigenvalue bound, \autoref{thmgeneralmult}. We present these special cases due to the importance of the operators as well as relative simplicity of the statements of the results. It is also worth noting that Hardy-Littlewood-P\'olya majorization can be used to further generalize all results for sums of eigenvalues to sums of any concave increasing function of eigenvalues (see \cite[Appendix A]{LS14}). Particularly interesting generalizations involve products of eigenvalues, partial sums of the spectral zeta function and trace of the heat kernel (cf. \cite[Theorem 1.1]{LS14}). Results for majorized sums can be added to virtually any bound in this paper. For example \autoref{thmgeneralmult} implies \autoref{thmgeneralmajorized}.

Our examples naturally split into plate-related higher order multipliers, and fractional probability related multipliers. The common theme is that we transform a domain using an invertible linear transformation on $\Rd$. Throughout this section $T$ will denote the transformation as well as its defining matrix.

\subsection{Clamped plate with tension}\label{mainclamped}
Consider the eigenvalue problem for the bi-Laplacian operator with tension and clamped boundary conditions
\begin{align*}
  \Delta^2u-\tau \Delta u&=\Gamma u &\text{ on }\Omega,\\
  u=\frac{\partial u}{\partial n}&=0 &\text{ on }\partial\Omega.
\end{align*}
Its eigenvalues $\Gamma_i$ correspond to frequencies of oscillations (energy levels) of a rigid plate $\Omega$ with clamped boundary. The constant $\tau$ corresponds to the ratio of the lateral tension to the flexural rigidity of the plate. Positive value of $\tau$ means that the plate is under tension, while negative $\tau$ indicates compression. Note that the Laplacian is a negative operator, hence $-$ sign in front of $\tau$ ensures that the higher the tension, the higher the eigenvalues. The eigenvalues $\Gamma_i$ were extensively studied both theoretically (e.g. Ashbaugh-Benguria \cite{AB95}, Kawohl-Levine-Velte \cite{KLV93}, Nadirashvili \cite{N95}, Payne \cite{P55}, Szeg\"o \cite{Sz50,Sz58}, Talenti \cite{T82}) and numerically (e.g. Kuttler-Sigillito \cite{KS80,KS81}, Leissa \cite{L69}, McLaurin \cite{McL68}, Shibaoka \cite{S57}, Wieners \cite{W96}). Abundance of numerical results clearly suggests theoretical difficulties in finding $\Gamma_i$. Indeed, one can explicitly find only the eigenvalues of balls.

We would like to emphasize, that our results naturally split into weaker but more general bounds obtained using classical frames, and stronger bounds requiring higher order frames. Whenever we write that $\Omega$ admits a $p$-frame, we mean that the isometry group of $\Omega$ is rich enough to generate $p$-frames. Note that classical $1$-frames are related to irreducibility of the isometry group of $\Omega$ via Schur's lemma (see e.g. Vale-Waldron \cite{VW04,VW05}). In \autoref{pexists} we give a list of domains admitting $p$-frames. In order to state the first theorem we need to define appropriate matrix norms.

\begin{definition}
  The Schatten norm of order $n$ of matrix $T$ is a sum of the $n$-th powers of the singular values of $T$. Equivalently
  \begin{align*}
    \|T\|_{n}^{n}=\tr\left(\sqrt{T^{\dagger}T}\right)^n,
  \end{align*}
  where $\dagger$ denotes the transpose and $\tr$ is the trace of a matrix. In particular when $n=2$ we obtain a more familiar Hilbert-Schmidt norm of $T$.
\end{definition}
 We will use the notation $\cdot|_{\tau,\Omega}$ to indicate that geometric/spectral quantity is evaluated on $\Omega$, with tension parameter $\tau$. 
\begin{theorem}\label{thmplate}
  Suppose $\Omega$ is a bounded domain in $\Rd$ with irreducible group of isometries (e.g. ball, regular polygon/polytope) and $T$ is a linear transformation. Then for the linearly transformed domain $T(\Omega)$:
  \begin{align}\label{plateweaker}
    \Gamma_1+\dots+\Gamma_n \Big|_{\tau C(T^{-1})/\|T^{-1}\|_2^2,T(\Omega)}\le \left.\frac{C(T^{-1})}{d}(\Gamma_1+\dots+\Gamma_n) \right|_{\tau,\Omega},
  \end{align} 
  where 
  \begin{align*}
    C(T^{-1})&=\|T^{-1}\|_4^4.
  \end{align*}

Furthermore if $\Omega$ admits $2$-frames (e.g. balls, regular polygons except for squares) then 
\begin{align}\label{platestronger}
  \Gamma_1+\dots+\Gamma_n \Big|_{\tau D(T^{-1})/\|T^{-1}\|_2^2,T(\Omega)}\le \left.\frac{D(T^{-1})}d(\Gamma_1+\dots+\Gamma_n) \right|_{\tau,\Omega},
  \end{align}
  where 
  \begin{align*}
    D(T^{-1})&=\frac{\|T^{-1}\|_2^4+2\|T^{-1}\|_4^4}{d+2}\le C(T^{-1}).
  \end{align*}
In both cases, equality holds when $T$ is a scalar multiple of an orthogonal matrix. 
 \end{theorem}
 Note that we need to vary $\tau$ on the left of \eqref{plateweaker} and \eqref{platestronger} to keep equality for all $T$ of the form $c \Id$ (simple scaling). We prove this result in \autoref{secbilap}.

 The inequality $D(T^{-1})\le C(T^{-1})$ (with equality if and only if $T$ is a scalar multiple of an orthogonal matrix) follows from quadratic-arithmetic mean inequality. However, that does not mean that \autoref{platestronger} is sharper than \eqref{plateweaker} when $\tau>0$. We do have a smaller constant on the right, but we also have a smaller tension parameter on the left. The inequalities are also not comparable when $\tau<0$, since we might have negative eigenvalues. 

 Finally, not every irreducible isometry group (symmetric domain $\Omega$) admits $2$-frames (see \autoref{pexists}). In particular, squares do not admit $2$-frames. Therefore \autoref{platestronger} cannot be used to bound plate eigenvalues of rectangles. This somewhat disappointing limitation is not just an artifact of our method. In \autoref{platenumerical} we show that \eqref{platestronger} is simply wrong for rectangles. In the same section we compare our bounds with known results for elliptical and triangular plates.

 One can view this theorem as a starting point for various natural simplifications. We state the most interesting corollaries for the $2$-frame case, but similar results can be obtained for classical frames and/or more general Fourier multipliers. Analogous simplifications for Laplace eigenvalue problems were developed in \cite{LS11plane} and \cite{LS11}.

In a slight expense of accuracy, we can avoid changing $\tau$ by fixing volume of the transformed domain. We get:
\begin{corollary}\label{corstrongplate}
  Suppose $\tau=0$, or $|\det T|=1$ and $\tau>0$. If $\Omega$ admits $2$-frames then
  \begin{align*}
    \left.\Gamma_1+\dots+\Gamma_n \right|_{\tau,T(\Omega)}\le \left.\frac{D(T^{-1})}d(\Gamma_1+\dots+\Gamma_n) \right|_{\tau,\Omega}.
  \end{align*}
\end{corollary}
The proof is essentially identical to the proof of \cite[Corollary 5]{LS11} and will be omitted.

Instead of the constants involving matrix norms, we can also express the $2$-frame case in a more geometric way. Let $I_{2p}(\Omega)$ denote the polar $2p$-moment of mass of $\Omega$ around the origin
\begin{align*}
  I_{2p}(\Omega)=\int_\Omega |x|^{2p} \,dx.
\end{align*}
See \autoref{secmoments} for details. So that $I_2$ is the polar moment of inertia of the domain and $I_0=V$ is its volume.
Then \autoref{lemmoments} gives the following equivalent restatement of \autoref{corstrongplate}.
\begin{theorem}\label{plateiso}
  Suppose $\Omega$ admits $2$-frames. If $\tau=0$ or $|\det T|=1$ and $\tau>0$, then
  \begin{align*}
    (\Gamma_1+\dots+\Gamma_n)V^{4/d} \Big|_{\tau,T(\Omega)}\cdot \left.\frac{V^{1+4/d}}{I_4}\right|_{T^{-1}(\Omega)}
  \end{align*}
  is maximal when $T$ is a scalar multiple of an orthogonal matrix.
\end{theorem}
Note that both factors above are scale invariant, and it is not possible to combine them into one simpler quantity on a single domain (see the discussion in \cite[Section III.B]{LS11}). When $n=1$ and $\tau=0$, minimization of the first factor (evaluated on $T(\Omega)$) reduces to the classical Rayleigh conjecture for the lowest plate eigenvalue, and balls should be the minimizers among all domains. However this is only known in dimensions $2$ (Nadirashvili \cite{N95}) and $3$ (Ashbaugh-Benguria \cite{AB95}). In higher dimensions partial results were obtained by Talenti \cite{T82} and Ashbaugh-Laugesen \cite{AL96}. The second, purely geometric factor in \autoref{plateiso} is maximal for balls, hence it can be viewed as a compensating factor, turning lower isoperimetric bound into a more general upper bound. 

  Finally, a simple relation between moments of mass of $T(\Omega)$ and $T^{-1}(\Omega)$ in two dimensions, see \autoref{momentstwo}, gives:

  \begin{corollary}\label{corplateiso}
  If planar $\Omega$ admits $2$-frames then \autoref{plateiso} simplifies to the following:
  \begin{align*}
    \left.(\Gamma_1+\dots+\Gamma_n)A^2\cdot \frac{A^3}{I_4}\right|_{\tau,T(\Omega)}
  \end{align*}
  is maximal when $T$ is a scalar multiple of an orthogonal matrix.

  In particular, among triangles, the quantity above is maximal for equilateral triangles. Similarly, among ellipses, the balls are maximizers.
\end{corollary}
Note that square might be a maximizer among all parallelograms in the above corollary, even though \autoref{platestronger} does not hold, since the corollary is weaker than \eqref{platestronger}. Indeed, neither $I_4$ nor eigenvalues for squares can be estimated using $2$-frames, and there could be some case-specific cancellation. 

Finally, \autoref{platestronger} extends to a related plate buckling problem (see \autoref{secbuckling} for detailed exposition). The goal is to find a critical value $\Lambda$ of the tension parameter $\tau$, that forces plate to buckle. The simplest, two-dimensional version of \autoref{thmbuckling} reads
\begin{theorem}
If planar $\Omega$ admits $2$-frames, then the buckling eigenvalue $\Lambda$ satisfies: 
  \begin{align*}
    \Lambda A \cdot\frac{AI_2}{I_4}\Big|_{T(\Omega)}
  \end{align*}
  is maximal when $T$ is a scalar multiple of an orthogonal matrix.
\end{theorem}
Again, the first (scale invariant) factor should be minimal for balls among all domains, but this P\'olya-Szeg\"o conjecture is still open.
\subsection{Bochner subordinators}
We also prove a general result for sums of eigenvalues of the generators of arbitrary subordinated Brownian motions. More precisely, for any complete Bernstein function $\Psi$ (see \cite{SSV12}), one defines a semigroup associated with a subordinated Brownian motion with generator $\Psi(-\Delta)$. Here we state theorems for the most often used cases, while proofs and general results are contained in \autoref{sec:fractional}. 

Perhaps the best known example of a subordinator is an $\alpha/2$-stable subordinator $\Psi(t)=t^{\alpha/2}$ with $0<\alpha<2$. The resulting operator is the fractional Laplacian $(-\Delta)^{\alpha/2}$ on a domain $\Omega$ with Dirichlet boundary condition outside of $\Omega$. This operator generates a semigroup of the symmetric $\alpha$-stable process killed while exiting $\Omega$. It is worth noting that this operator is not the same as the $\alpha/2$-power of the Dirichlet Laplacian. Hence its spectrum is not just a power of the Dirichlet spectrum (see Chen-Song \cite{CS05}).

The spectral theory of the fractional Laplacian is of great interest. The smallest eigenvalue of this operator satisfies Rayleigh-Faber-Krahn isoperimetric inequality (Ba\~nuelos \textit{et al.} \cite{BLM01}). More precisely, among all domains with fixed area/volume, ball has the least smallest eigenvalue. Various spectral properties of the fractional Laplacian have been studied by e.g. Ba\~nuelos-Kulczycki \cite{BK06,BK08}, Chen-Song \cite{CS05,CS06}, Frank-Geisinger \cite{FG11}, Yolcu-Yolcu \cite{YY12,YY13}. For a comprehensive overview see \cite[Chapter 4]{BBKRSV}. Numerical results have also been obtained: for an interval by Kulczycki \textit{et al.} \cite{KKMS10}, for a ball by Dyda \cite{D12}. Clearly, it is not possible to find exact eigenvalues for balls, or even intervals.

As in the plate case, we prove bounds for sums of eigenvalues involving norms of the transformation $T$. Here we only state the simplified isoperimetric upper bound (cf. \autoref{plateiso})
\begin{theorem}\label{thmfractional}
  (Fractional Laplacian) For any $\Omega$ with irreducible isometry group and $0<\alpha\le 2$, the eigenvalues of $(-\Delta)^{\alpha/2}$ satisfy:
  \begin{align*}
    \left.\left(\lambda_1+\dots+\lambda_n\right)^{2/\alpha} V^{2/d}\right|_{\alpha,T(\Omega)}\cdot \left.\frac{V^{1+2/d}}{I_2}\right|_{T^{-1}(\Omega)}
  \end{align*}
  is maximal when $T$ is a scalar multiple of an orthogonal matrix.
\end{theorem}
The special case $\alpha=2$ reduces to \cite[Corollary 2]{LS11} (upper bound for the sums of Dirichlet eigenvalues of the classical Laplacian). As in the plate case we get a significantly simpler statement in two dimensions (cf. \autoref{corplateiso}).
\begin{corollary} In dimension 2
  \begin{align*}
    \left.\left(\lambda_1+\dots+\lambda_n\right)^{2/\alpha}A\cdot \frac{A^2}{I_2}\right|_{\alpha,T(\Omega)}
  \end{align*}
  is maximal when $T$ is scalar multiple of an orthogonal matrix.

  In particular, disk is the extremizer among ellipses, square among parallelograms and eqiulateral among all triangles.
\end{corollary}

We can also use a $p$-frame-like identity with $p=1/2$ (even though $p$ is an integer in true frames), available for the isometry group of the disk, to get the following improved result for ellipses, involving the isoperimetric ratio $A/L^2$:
\begin{theorem}\label{thmperimeter}
  Let $E$ be an ellipse with perimeter $L$ and area $A$. Then for $\alpha\le 1$,
  \begin{align*}
    \left.(\lambda_1+\dots+\lambda_n)^{2/\alpha} A\cdot \frac{A}{L^2}\right|_{\alpha,E}
  \end{align*}
  is maximal when $E$ is a disk.
\end{theorem}
Note that $\lambda_1 A^2/L^2$ for the Dirichlet Laplacian is a relatively well behaved, bounded quantity, but with degenerate domains as extremizers (Polya \cite{P60} and Makai \cite{M62}), instead of balls. Therefore, it would be interesting to check what is the extremal domain in our case, among all convex domains, even just for the first eigenvalue. 

Fix mass $m\ge0$. Relativistic subordinator $\Psi(t)=\sqrt{m^2+t}-m$ leads to another important operator: the Klein-Gordon operator  $\sqrt{m^2-\Delta}-m$ that models relativistic particles in quantum mechanics (see Harrell-Yolcu \cite{HY09} and Yolcu \cite{Y09}). This example requires more care, as the relativistic constant must change with the linear transformation, as did the tension constant for plates. For that reason we use $\cdot|_{m,\Omega}$ to indicate that a quantity is evaluated on $\Omega$ with mass $m$.
\begin{theorem}\label{thmrelativistic} (Klein-Gordon operator)
  Sums of eigenvalues of $\sqrt{m^2-\Delta}-m$ satisfy
  \begin{align*}
    \lambda_1+\dots+\lambda_n \Big|_{md/\|T^{-1}\|_{2},T(\Omega)}\le \left.\frac{\|T^{-1}\|_{2}}{\sqrt{d}}\left(\lambda_1+\dots+\lambda_n\right)\right|_{m,\Omega}.
  \end{align*}

  Furthermore, if $|\det T|=1$, then
  \begin{align*}
    \left.\left(\lambda_1+\dots+\lambda_n\right)^2 V^{2/d}\right|_{m,T(\Omega)}\cdot \left.\frac{V^{1+2/d}}{I_2}\right|_{T^{-1}(\Omega)}
  \end{align*}
  is maximal when $T$ is an orthogonal transformation.
\end{theorem}
Note that as in the fractional Laplacian case, we also have a simplified two-dimensional version, without the inverse transformation $T^{-1}$. Finally, when $m=0$ this last result reduces to the fractional Laplacian with $\alpha=1$.

Our final example involves a generic complete Bernstein function $\Psi$, and it unifies \autoref{thmrelativistic} and \autoref{thmfractional} (additional assumption $|\det T|=1$ is irrelevant for the fractional Laplacian due to homogeneity of $\Psi$). For any concave $\Psi$ (in particular any Bernstein function) there exists $0<\beta\le1$ such that $\Psi(t/c)\le c^{-\beta}\Psi(t)$ for any $c>1$. Therefore
\begin{theorem}\label{thmsubor}
  If $|\det T|=1$, then the eigenvalues of $\Psi(-\Delta)$ satisfy
  \begin{align*}
    \lambda_1+\dots+\lambda_n \Big|_{\Psi,T(\Omega)}\le \left.\frac{\|T^{-1}\|_{2}^{2\beta}}{d^\beta}\left(\lambda_1+\dots+\lambda_n\right)\right|_{\Psi,\Omega}.
  \end{align*}
\end{theorem}
Note that for fractional Laplacian we have $\beta=\alpha/2$ and for the relativistic operator $\beta=1/2$.

\section{General approach}\label{qforms}
Now we develop the general context that will be used to handle all operators described in the main results section. Note that these are only examples of what we can work with, albeit the best known ones. 

\subsection{Quadratic forms in frequency domain}
Define the Sobolev space $H^\beta(\Rd)$ as the subset of $L^2(\Rd)$, such that the Fourier transform of a function $f\in L^2(\Rd)$ satisfies 
\begin{align*}
  \|f\|_\beta^2 = \int_\Rd |\xi|^{2\beta} |\widehat f|^2 d\xi<\infty.
\end{align*}
Then $\|f\|_\beta$ is the Sobolev seminorm and $\|f\|_0$ is the standard $L^2$ norm. Note that when $\beta$ is an integer, this coincides with the classical definitions involving derivatives of $f$. The subspace $H_0^\beta(\Omega)$ of $H^\beta(\Rd)$ is a closure of $C_c^\infty(\Omega)$ in the Sobolev norm $\sqrt{\|f\|_0^2+\|f\|_\beta^2}$. See \cite{DPV} for a detailed discussion about various ways to define fractional Sobolev spaces, their equivalence and embedding properties.

Consider a weakly defined self-adjoint linear operator $A$ with the domain $H^\beta_0(\Omega)$ (possibly fractional Sobolev space). That is, for any $u,v\in C^{\infty}_c(\Omega)$ we define the symmetric quadratic form corresponding to $A$
\begin{align*}
  Q_A(u,v)=\int_\Omega (Au)(x) v(x) dx=\int_\Omega u(x) (Av)(x) dx.
\end{align*}
This form can then be extended to $u,v\in H^\beta_0(\Omega)$. In many classical cases, including Dirichlet Laplacian, one can use Fourier transform to rewrite the quadratic form in the frequency domain. Note that the assumption that $u\in H_0^\beta(\Omega)$ (informally: function $u$ is zero on the boundary (or outside $\Omega$ for fractional cases) as are all its derivatives up to order $\beta$) means that extending with $u=0$ outside $\Omega$ gives a function in $H^\beta(\Rd)$. Parseval's theorem can now be used to justify a frequency domain formula for the quadratic form
\begin{align}\label{Qfourier}
  Q_A(u,v)=\int_\Rd s(\xi) \widehat{u}(\xi) \overline{\widehat{v}(\xi)} d\xi,
\end{align}
where $s(\xi)$ is called the Fourier multiplier (or symbol) corresponding to $A$. In particular, classical Laplacian corresponds to $s(\xi)=|\xi|^2$, while powers of the Laplacian with Dirichlet boundary conditon can be defined as Fourier multipliers $|\xi|^{2\beta}$ with the domain $H_0^\beta(\Omega)$.

In this paper we only consider multipliers $s(\xi)=f(|\xi|^2)$ with at most polynomially growing bounded below function $f$ (in particular any linear combination of power functions with positive coefficient for the highest power). Since $|\xi|^2$ is the Fourier multiplier for the Dirichlet Laplacian on $\Omega$, we call these Laplace related multipliers. The multiplier only depends on the length of $\xi$, hence it is invariant under isometries in frequency domain. This corresponds to operators $A$ invariant under isometries of physical space, just like the classical Laplacian. We are interested in the interaction of linear transformations with these operators. In particular, how does the spectrum of the operator change under such transformation?

Consider an invertible linear map $T:\Rd\to\Rd$. We will abuse the notation slightly, by using $T$ to denote both the linear transformation, and the matrix that defines it. The Fourier transform almost commutes with $T$, since
\begin{align*}
  \widehat{u\circ T}(\xi)&=\int_\Rd u(Tx) \exp(-2\pi i \xi\cdot x)dx=\int_\Rd u(y)\exp(-2\pi i \xi\cdot T^{-1}y)|\det T|^{-1}dy
  \\&=
  \int_\Rd u(y)\exp(-2\pi i T^{-\dagger}\xi\cdot y)|\det T|^{-1}dy=|\det T|^{-1}\widehat u(T^{-\dagger}\xi),
\end{align*}
where $T^{-\dagger}$ denotes the transposed inverse of the matrix $T$.

This allows us to evaluate the quadratic form on the linearly transformed functions
\begin{align}
  Q_A(u\circ T,u\circ T)&=\int_\Rd f(|\xi|^2)|\widehat{u\circ T}|^2d\xi=|\det T|^{-2}\int_{\Rd} f(|\xi|^2)|\widehat u(T^{-\dagger}\xi)|^2d\xi\nonumber
  \\&=
  |\det T|^{-2}\int_{\Rd} f(|T^{\dagger}\xi|^2)|\widehat u(\xi)|^2|\det T^{\dagger}|d\xi\nonumber
  \\&=
  |\det T|^{-1}\int_{\Rd}f(|T^{\dagger}\xi|^2)|\widehat u|^2d\xi.\label{QT}
\end{align}

Note that $f$ is bounded below and grows at most polynomially, hence $f(|\xi^2|)\approx f(|T^\dagger \xi|^2)$ and the last integral converges. Furthermore, taking $f(t)=t^{2\beta}$ implies that $u\circ T$ belongs to $H^\beta(\Rd)$ whenever $u\in H^\beta(\Rd)$. Finally, if $C_c^\infty(\Omega)\ni u_n\to u\in H^\beta_0(\Omega)$ in the $H^\beta(\Rd)$ norm, then $u_n\circ T\to u\circ T$ by a similar calculation. Hence $u\circ T$ is in the domain of the operator $A$ if and only if $u$ is in its domain.

If $T$ is an orthogonal matrix, we get $Q_A(u\circ T,v\circ T)= Q_A(u,v)$ (showing invariance of $Q_A$ under isometries). While for $T=c\Id$ and $f$ homogeneous of degree $\beta$ we get $Q_A(u\circ T,u\circ T)=c^{2\beta-d}Q_A(u,u)$ (use \eqref{QT} with the particular choice of $f$ and $T$) . For general matrices $T$, one cannot immediately simplify the quadratic form. However, $p$-frames described in the next section allow us to simplify many interesting cases.

Note also that $L^2$ norm (equivalent to $\beta=0$) satisfies 
\begin{align*}
  \|u\circ T\|_0^2=|\det T|^{-1}\|u\|_0^2.
\end{align*}

From now on, we also assume that $H^\beta_0(\Omega)$ embeds compactly into $L^2(\Omega)$. When $\beta=1$, this is true for any bounded domain $\Omega$ (Rellich's theorem), while $H^\beta_0(\Omega)$ with $\beta>1$ imbeds isometrically into $H^1_0(\Omega)$, which then imbeds compactly into $L^2$. For the case $\beta<1$ we refer the reader to \cite{DPV}. Boundedness below of $f$ implies that for some $c$ the operator $A+c$ is elliptic. This, together with compact embedding implies discreteness of the spectrum (see Osborn-Babu\v{s}ka \cite{BO91}, or Blanchard-Br\"uning \cite{BB92}, cf. Section 4 of Laugesen \cite{Lnotes}). 

\subsection{Geometric averaging and eigenvalue bounds}
In order to find the smallest eigenvalue one can minimize the Rayleigh quotient (see e.g. Bandle \cite{B90})

\begin{align*}
  R[u]&=\frac{Q_A(u,u)}{\|u\|_0^2},\\
  \lambda_1&=\inf_u R[u].
\end{align*}
Similarly one can find the sum of consecutive eigenvalues
\begin{align*}
  \sum_{i=1}^n \lambda_i &= \inf\left\{ \sum_{i=1}^n R[u_i]:\; u_i \text{ mutually $L^2$-orthogonal}\right\}.
\end{align*}

Consider highly symmetric domain $\Omega$ (with irreducible symmetry group $G$) and its linear image $T(\Omega)$. Suppose $u_i$ are the orthonormal eigenfunctions for $\lambda_i(\Omega)$. Due to symmetry of $\Omega$, for any isometry $U\in G$, functions $u_i\circ U$ are also orthonormal eigenfunctions for $\Omega$. Usually, these are not the same eigenfunctions, but they belong to the same eigenvalues, and are still orthogonal.

Consider functions $v_i: T(\Omega)\to \Real$ defined by $v_i(x)=u_i\circ U\circ T^{-1}(x)$. Linear transformations preserve $L^2$ orthogonality, and zero boundary condition. More formally, $u_i\circ U\circ T^{-1}\in H^{\beta}_0(T(\Omega))$. We get
\begin{align*}
  \lambda_1\Big|_{T(\Omega))}\le R[v_1]=\frac{Q_A(u\circ U\circ T^{-1},u\circ U\circ T^{-1})}{\|u\circ U\circ T^{-1}\|_0^2}=\frac{
  \int_{\Rd}f(|T^{-\dagger}U^\dagger \xi|^2)|\widehat u|^2d\xi
  }{\|u\|_0^2}
\end{align*}
and similarly for sums of eigenvalues. But the left side is independent of $U$, hence we can average over the group $G$ of symmetries of $\Omega$ (or any of its subgroups). Therefore
\begin{align}\label{average1}
  \lambda_1\Big|_{T(\Omega)}\le
  \frac{1}{|G|}\sum_{U\in G}\frac{\int_{\Rd}f(|T^{-\dagger}U^\dagger \xi|^2)|\widehat u|^2d\xi
  }{\|u\|_0^2}
\end{align}
Similarly, when the $u_i$ are orthonormal,
\begin{align}\label{averagesum}
  \lambda_1+\dots+\lambda_n\Big|_{T(\Omega)}\le
  \frac{1}{|G|}\sum_{U\in G}\sum_{i=1}^n\int_{\Rd}f(|T^{-\dagger}U^\dagger \xi|^2)|\widehat u_i|^2d\xi.
\end{align}
In the simplest case $f(t)=t$, the averaging of this type was used in \cite{LS11} to obtain sharp bounds for eigenvalues of the Laplacian, by separating $|\xi|^2$ from $T$. This required irreducibility of $G$ as a subgroup of the orthogonal group on $\Rd$, but $u_i$ did not need to be known explicitly.

In \autoref{pmatrix} we generalize the theory of $p$-frames so that optimal averaging can be applied to other power functions $f$, leading to the plate results described in \autoref{mainclamped}. Note that averaging can always be performed if $G$ is the full orthogonal group (we average with respect to Haar measure on this group). See \autoref{thmperimeter} for an application of this approach. Note also that the Haar measure averaging generalizes to nonlinear transformations of balls \cite{LS14,LS14b}, assuming $f(t)=t$. 

\subsection{A bound for general Fourier multipliers}

Now we introduce a classical $1$-frame based, general approach that allows us to handle almost arbitrary multipliers $f(|\xi|^2)$. Recall the Schatten $1$-norm
\begin{align*}
  \|T\|_1 = \tr \sqrt{T T^\dagger},
\end{align*}
where $\dagger$ denotes the transpose of a matrix.
Assume $M$ is symmetric and nonnegative definite, so that $\|M\|_1 = \tr M$. Write the spectral decomposition $M=VEV^{-1}$ with orthonormal $V$ and diagonal $E$. Define
\begin{align*}
  \Phi(M) = V\Phi(E)V^{-1},
\end{align*}
with entry-wise action of $\Phi$ on the diagonal matrix $E$.
For nonnegative $\Phi$ we have
\begin{align*}
  \|\Phi(M)\|_1 = \tr \Phi(M) = \tr \Phi(E).
\end{align*}
For arbitrary $\Phi$ the trace still makes sense (sum of the diagonal elements or eigenvalues is well defined), though the Schatten $1$-norm is no longer equal to the trace. Write $f=\Phi_1-\Phi_2$ and for arbitrary matrix $T$ define
\begin{align}\label{newmult}
  F[\Phi_1,\Phi_2,T](t) = \frac{1}{d}\tr \Phi_1(t T^{-1}T^{-\dagger})-\Phi_2\left(\frac{t}{d}\tr T^{-1}T^{-\dagger}  \right)
\end{align}
Note that this quantity strongly depends on the choice of $\Phi_i$. However, we have the following weak form of invariance. For any linear function $l$
\begin{align*}
  F[\Phi_1+l,\Phi_2+l,T](t)=F[\Phi_1,\Phi_2,T](t).
\end{align*}
We will use this property in \autoref{secnontight} in the proof of the next theorem.

Any $f\in C^{1,1}_{loc}(0,\infty)$ can be decomposed as a difference of two convex functions \cite{H59,BB11}, therefore the following theorem can be applied to virtually arbitrary multipliers. Due to a complicated statement of the result, we employ the notation $\cdot|_{f,\Omega}$, which indicates that the quantities are evaluated on $\Omega$ and the operator has the Fourier multiplier $f(|\xi|^2)$.
\begin{theorem}\label{thmgeneralmult}
  Suppose it is possible to write $f(t)=\Phi_1(t)-\Phi_2(t)$, where $\Phi_i$ are convex. Let $T$ be an invertible linear transformation and $\Omega$ be a domain with irreducible isometry group. Then
  \begin{align*}
    \lambda_1+\dots+\lambda_n \Big|_{f,T(\Omega)}\le \lambda_1+\dots+\lambda_n\Big|_{F[\Phi_1,\Phi_2,T],\Omega},
  \end{align*}
  with equality when $T$ is a multiple of an orthogonal matrix. Furthermore, if $f=\Psi\circ \Phi$ with concave $\Psi$ and convex $\Phi$ then
  \begin{align*}
    \lambda_1+\dots+\lambda_n \Big|_{f,T(\Omega)}\le \lambda_1+\dots+\lambda_n\Big|_{\Psi\circ F[\Phi,0,T],\Omega}.
  \end{align*}
\end{theorem}

The proof is postponed to \autoref{secnontight}.
Note that this result is very broad, but awkward to apply, in practice. The Fourier multiplier on the right is not the same as the multiplier on the left. Thus we compare eigenvalues of different operators. Cases listed in the previous section use special form of $f$ to simplify the statement of this theorem. Also, \autoref{thmconcave} presents a much simpler looking case $\Phi_1(t)=0$ (or equivalently $\Phi(t)=t$). Finally, if $f\in C^2(0,\infty)$ and $f''(t)\ge -c$, then we can convexify $f$ by taking $\Phi_1(t)=f(t)+c t^2$ and $\Phi_2(t)=ct^2$. This approach applies to any polynomial $f(t)$ of degree at least $3$ with positive leading term. Furthermore, higher order frames lead to improved results, as described in \autoref{secpolylap}.

As with all our results, Hardy-Littlewood-P\'olya majorization (cf. \cite[Theorem 1.1]{LS14}) implies:
\begin{corollary}\label{thmgeneralmajorized}
  Let $f$ and $\Omega$ be as in \autoref{thmgeneralmult}. For any concave increasing function $\psi$
  \begin{align*}
    \psi(\lambda_1)+\dots+\psi(\lambda_n) \Big|_{f,T(\Omega)}\le \psi(\lambda_1)+\dots+\psi(\lambda_n)\Big|_{F[\Phi_1,\Phi_2,T],\Omega},
  \end{align*}
\end{corollary}

\section{Generalized $p$-frames from finite symmetry groups}\label{pmatrix}
\subsection{General setup}
Bachoc and Ehler \cite{BE13} considered the following generalization of the classical frames.
A set of vectors $\{v_i\}_{i=1}^N$ in $\Rd$ is called a tight $p$-frame, for an integer $p$, if
\begin{align}\label{pframe}
  \frac{1}{N}\sum_i \langle x, v_i\rangle^{2p} = \frac{(1/2)_p}{(d/2)_p} |x|^{2p}|v_i|^{2p},
\end{align}
where $(\cdot)_p$ denotes the rising factorial.  In the same paper, authors also consider a more general tight $p$-fusion frames. The inner products $\langle x, v_i\rangle$ can be viewed as a projection of $x$ onto vectors $v_i$. Replace these rank 1 projections with arbitrary projections on subspaces $V_i$ to obtain tight $p$-fusion frames. We are interested in a special case of fixed dimension $V_i$, say $dim(V_i)=k$. The set of the projections $P_{V_i}$ (given as orthogonal matrices) is a tight $p$-fusion frame if for any $x$
\begin{align}\label{PVframe}
  \frac{1}{N}\sum_i |P_{V_i}x|^{2p} = \frac{(k/2)_p}{(d/2)_p} |x|^{2p}.
\end{align}

Note that when $p=1$, the definitions above reduce to classical tight frames and classical tight fusion frames (see Casazza \textit{et al.} \cite{CKL,CFMWZ}.

Let $G$ be a finite group of symmetries of a polytope in dimension $d$ (subgroup of the orthogonal group $O(d)$). This group can be used to generate projections $P_{V_i}$ using the orbit of a single subspace $V$. The orbit $U(V)$, for $U\in G$, may contain the same subspace many times, resulting in $V_i=V_j$ for some $i$ and $j$. This is not explicitly forbidden in the definition of the tight $p$-fusion frame. In fact, the union of two tight $p$-fusion frames is again a tight $p$-fusion frame. One can even combine frames with different subspace dimensions, resulting in a new frame constant. Note also that different choices of $V$ may result in different lengths of the orbits. Nevertheless, group $G$ generates a tight $p$-fusion frame if
\begin{align*}
  \frac{1}{|G|} \sum_{U\in G} |P_{U(V)}(x)|^{2p}=\frac{1}{|G|} \sum_{U\in G} |UP_{V}U^{-1}x|^{2p}=\frac{(k/2)_p}{(d/2)_p} |x|^{2p}.
\end{align*}

The existence of tight $p$-fusion frames generated by $G$ is equivalent to uniqueness of the invariant polynomials of degree $2p$ for a given group $G$ \cite[Theorem 6.1]{BE13}. In particular, if $G$ is not irreducible, then there exists an invariant hyperplane (orbit not spanning the whole $\Rd$), hence also invariant polynomial of degree 1. Its square is also invariant. But $|x|^2$ is a different invariant of degree 2, since it is invariant even under the full orthogonal group. This shows necessity of the uniqueness for $p=1$ (classical frames).

The following lemma states sufficiency condition for existence of tight $p$-frames. The argument used in its proof will lead to generalized tight $p$-frames later in this section.
\begin{lemma}[\protect{Special case of \cite[Theorem 6.1]{BE13}}]\label{framelemma}
  If a finite symmetry group $G$ acting as isometries of $\Rd$ has unique invariant polynomial of degree $2p$, then any orbit of $G$ generates a tight $q$-frame for any $q\le p$.
\end{lemma}
\begin{proof}
  Group $G$ cannot have fundamental invariant polynomials of degree $2q>2$ with $q\le p$. Otherwise we could get more than one invariant polynomial of degree $2p$. Indeed, any subgroup of the orthogonal group has invariant quadratic $|x|^2$. If there is a fundamental invariant $f$ of degree $2q<2p$, then $|x|^{2p-2q}f$ would be a $2p$-degree invariant, different then $|x|^{2p}$. Hence there is a unique invariant polynomial of degree $2q$, equal $|x|^{2q}$, for any $q\le p$.

We want to prove the tight $q$-frame identity with $q\le p$ for an orbit of $G$. That is
\begin{align}\label{cdotG}
   F(v):=\frac{1}{|G|}\sum_U \langle x,Uv\rangle^{2q} = \frac{(1/2)_q}{(d/2)_q} |x|^{2q}|v|^{2q}.
\end{align}
Note that the right side is an invariant polynomial in $v$. The inner product on the left is linear in $v$, hence the left side is a polynomial of degree $2q$ in $v$. Suppose $U'\in G$.
\begin{align}
  F(U'v)=\frac{1}{|G|}\sum_U \langle x, UU'v\rangle^{2q}=\frac{1}{|G|}\sum_{UU'} \langle x, UU'v\rangle^{2p}=\frac{1}{|G|}\sum_U \langle x, Uv\rangle^{2q}=F(v),
\end{align}
using group property of $G$. Hence $F(v)$ is an invariant polynomial of degree $2q$. Therefore it must equal $c|v|^{2q}$. The left side and the right side differ by a constant (in $v$, may still dependent on $x$), by uniqueness.
  
However $U$ is an isometry, hence $\langle x, Uv\rangle=\langle U^{-1}x,v\rangle$. Therefore the same argument shows that both sides differ by a constant dependent only on $v$. We see that
\eqref{cdotG} holds with an unknown constant (independent of $x$ and $v$). This constant is the same for any tight $p$-frame \cite[Section 4]{BE13} and is computed in \cite[Remark 5.2]{BE13}. We compute the same constant in \autoref{secfp}, as a special case of the polynomial $F_p$ from \autoref{corframes}.
\end{proof}

\subsection{Generalized tight $p$-frames}\label{subspmatrix}
The following reduction is crucial in working with classical frames \cite{Ba10,BBCE} as well as $p$-frames \cite{BE13}:
\begin{align*}
  \langle x,Uy\rangle^2=\tr(P_x UP_yU^{-1})=\tr(P_x P_{Uy})=\langle P_x,P_{Uy}\rangle_{HS}.
\end{align*}
where $P_x=xx^\dagger$ is a projection matrix, and $\langle\cdot,\cdot\rangle_{HS}$ is the Hilbert-Schmidt inner product of matrices.

A tight $2$-frame identity, using above notion, states that
\begin{align*}
  \frac{1}{|G|}\sum_U \langle P_x,P_{Uy}\rangle_{HS}^2  = \frac{3}{d(d+2)}|P_x|_{HS}^2|P_y|_{HS}^2.
\end{align*}
Therefore, matrices $P_{Uy}$ form a tight frame in the space of rank $1$ matrices equipped with Hilbert-Schmidt inner product. We will show that the same matrices form a tight frame in the space of all symmetric matrices. 
\begin{theorem}\label{theoremframes}
  Let $G$ be a group of symmetries with unique invariant polynomial of degree $2p$. Then for any symmetric matrix $M$, projections $P_{Uy}$ formed by the orbit $Uy$ satisfy
  \begin{align}\label{matrixframe}
    \frac{1}{|G|}\sum_U \langle M,P_{Uy}\rangle_{HS}^p  = F_{p}(\sigma(M))|P_y|_{HS}^p=F_{p}(\sigma(M))|y|^{2p},
  \end{align}
  where $F_{p}(\sigma(M))$ is a homogeneous, symmetric polynomial of degree $p$, and $\sigma(M)$ is the multiset of eigenvalues of $M$ (with repeated elements for multiple eigenvalues).
\end{theorem}
\begin{remark}
  Note that $\langle M,P_{Uy}\rangle_{HS}=\langle U^{-1}MU,P_{y}\rangle_{HS}$. Hence orbit of $M$ under conjugation with group $G$ forms a tight-frame-like object. However, even full orthogonal group acting by conjugation on matrices cannot change the spectrum of the matrix. Hence any symmetric polynomial in $\sigma(M)$ is invariant under this action. Therefore, polynomials $F_p$ are not just $|\sigma(M)|^p$, as is the case for the vector based actions (cf. \autoref{pframe}).
\end{remark}
\begin{proof}
  Note that the left side of \eqref{matrixframe} is a polynomial of degree $2p$ in $y$. It is also invariant under $G$. As in \autoref{framelemma}, we get that \eqref{matrixframe} holds up to a constant $F_p(M)$ depending only on $M$. However, the left side is a polynomial of degree $p$ in the entries $m_{ij}$ of the matrix $M$. Therefore, the constant term $F_p(M)$ is a polynomial in the entries $m_{ij}$ of $M$. This polynomial is obviously homogeneous. We need to show that it only depends on eigenvalues as a multiset (counting multiplicities). That is $F_p$ is a symmetric homogeneous polynomial.

Start by writing the spectral decomposition of $M$
\begin{align*}
  M=VEV^{-1},
\end{align*}
with orthonormal $V$ and diagonal $E$. For any vector $y$
  \begin{align*}
    \langle VEV^{-1},P_{y}\rangle_{HS}=\tr(VEV^{-1}P_{y})=\tr(EV^{-1}P_{y}V)=
    \langle E,P_{V^{-1}y}\rangle_{HS}.
  \end{align*}

Consider Haar measure $\mu$ on the orthogonal group, that is normalized to $\mu(O(d))=1$. We have
  \begin{align*}
F_{p}(M)|y|^{2p}&=\int_{O(d)} F_{p}(M)|Wy|^{2p}\,d\mu(W)=\int_{O(d)} \frac{1}{|G|}\sum_U \langle M,P_{UWy}\rangle_{HS}^p \,d\mu(W)
\\&=
\frac{1}{|G|}\sum_U \int_{O(d)} \langle E,P_{V^{-1}UWy}\rangle_{HS}^p \,d\mu(W)
  \end{align*}
  But Haar measure is invariant under actions of the group it is defined on, hence we can substitute $V^{-1}UW\to W$ to get
  \begin{align}\label{intOd}
    F_{p}(M)|y|^{2p}=\int_{O(d)}  \langle E,P_{Wy}\rangle_{HS}^p \,d\mu(W)
  \end{align}
  Therefore polynomial $F_p(M)$ depends only on the eigenvalue matrix $E$. However, we can use permutation matrices (subgroup of $O(d)$) to permute the diagonal entries in $E$ (spectral decomposition is unique up to the order of the diagonal entries in $E$). Hence $F_p$ only depends on eigenvalues as a multiset (it is symmetric).
\end{proof}
For arbitrary matrix $T$, define its singular values $s(T)$ as the multiset of eigenvalues of the matrix $\sqrt{T^\dagger T}$.  
\begin{corollary}\label{corframes}
  For any matrix $T$ with $d$ columns (possibly rectangular), and any group $G$ with unique invariant polynomial of degree $2p$
  \begin{align*}
    \frac{1}{|G|}\sum_U |TUx|^{2p}  = F_{p}(s^2(T))|x|^{2p}.
  \end{align*}
\end{corollary}
\begin{proof}
  Note that
  \begin{align*}
    |TUx|^2=Tr(TUx(TUx)^\dagger)=Tr(TUxx^\dagger U^{-1}T^{\dagger})=Tr(T^\dagger T U P_xU^{-1})=\langle T^\dagger T,P_{Ux}\rangle_{HS},
  \end{align*}
  and $T^\dagger T$ is symmetric.
\end{proof}
\begin{corollary}
  In the theorem and the corollary above the finite group $G$ can be replaced with full orthogonal group $O(d)$. This also means that the sum is replaced with the Haar measure integral. Finally, any restriction based on uniqueness of the invariant polynomials no longer applies. Hence $G=O(d)$ can be used with arbitrary value of $p$ in any dimension.
\end{corollary}

\subsection{Groups with unique invariant polynomial of degree $2p$ (admitting $p$-frames)}\label{pexists}
Existence of the invariant polynomials is a very well understood subject. In particular one can easily find a table of the degrees of fundamental invariants for reflection groups (see Benson and Grove \cite[Table 7.1]{BG}). In case a symmetry group is missing from a table (e.g. groups of rotations, without reflections) one can use Molien's series \cite[Theorem 7.4.4 with examples]{BG} to check the number of linearly independent invariants of a given degree. See also \autoref{secsage} for SAGE implementation of a Molien series. Below we list most of the interesting examples.

As we already pointed out, every subgroup of the orthogonal group has a quadratic fundamental invariant $|x|^2$. The full symmetry group of the tetrahedron, cube and octahedron (as well as their higher dimensional equivalents, Coxeter groups $A_n$, $C_n$ and $D_n$) have a fundamental invariant of degree $4$. Therefore the uniqueness fails in these cases, since $|x|^4$ is also invariant. Hence those groups do not admit tight $p$-frames with $p>1$. Nevertheless, in \autoref{secnontight} we will construct non-tight $p$-frames for these groups, and use them to get bounds for Laplace-like operators. 

In dimension 3, only icosahedral group (Coxeter $H_3$) does not have a fundamental invariant of degree $4$, but it does have one of degree $6$. Hence icosahedral symmetry admits tight $2$-frames, but not $3$-frames. The rotation group of the icosahedron (subgroup of index $2$ of $H_3$) does not have a degree $4$ fundamental invariant either, as we check using SAGE in \autoref{secsage}.

In dimension 4, the group of symmetries of a 24-cell (Coxeter $F_4$) admits $2$-frames, while the group of symmetries of 120-cell (Coxeter $H_4$) allows for $5$-frames. Finally, Coxeter groups $E_6$ and $E_7$ admit $2$-frames and $E_8$ allows for $3$-frames.

All these examples show that in dimensions at least $3$, there are few finite groups admitting higher order frames, possibly none. However, one can always get an infinite frame using full orthogonal group (symmetries of a ball) and Haar measure integral instead of summation.

Dimension $2$ is the most interesting, since the group of symmetries $I_2(n)$ of a regular $n$-gon has only two fundamental invariants. One of course of degree $2$, the other of degree $n$. Therefore $I_2(n)$ admits $p$-frames for any $p<n/2$. Furthermore, if $n$ is odd, we also get $p$-frames for any $p<n$. In particular symmetry of the equilateral triangle allows for $2$-frames, while regular pentagon admits $4$-frames. Somewhat surprisingly, symmetry of the square is not enough to allow even $2$-frames. On the other hand, the fact that squares and cubes do not admit $2$-frames is equivalent to Pythagorean theorem holding only with squares of the lengths of the sides of a right triangle (simplex).

\section{Polynomials $F_p$}\label{secfp}
In this section we derive a formula for the polynomials $F_p$ used in the statement of \autoref{theoremframes} and \autoref{corframes}. We start with formula \eqref{intOd} with unit vector $y$ and diagonal matrix $E$ with entries $e_1,\dots,e_d$. It is enough to consider positive $e_i$, and for convenience we choose to work with matrix $E^2$, instead of $E$.
\begin{align}\label{oversphere}
  F_{p}(E^2)=\int_{O(d)}  \langle E^2,P_{Uy}\rangle_{HS}^p ,d\mu(U)=\int_{O(d)} |EUy|^{2p}\,d\mu(U)=\frac{1}{|S^{d-1}|}\int_{S^{d-1}} |E\theta|^{2p} dS(\theta),
\end{align}
since Haar measure on $O(d)$ induces a uniform measure on the sphere. In \cite{BE13}, authors used rotational invariance of the Laplacian, and the fact that the Laplacian restricted to a subspace is still the Laplacian to handle matrices $E$ with $0$'s and $1$'s on the diagonal (subspace projections). In general one might try to turn the integral into an integral in spherical coordinates, however the calculations become very tedious in higher dimensions. Furthermore, we would like to also find a formula for $F_p$ that would not require the singular value decomposition of the transformation. Such a formula would be easier to apply to symbolic matrices, for which singular value decomposition is not easy to find.

Direct, polar coordinates based approach in two dimensions has an advantage when we take the whole orthogonal group as $G$. That is when $\Omega$ is a disk, and $T(\Omega)$ is an ellipse. In this case constants $F_p$ can be evaluated even for \textit{non-integer} values of $p$, albeit in terms of a rather complicated combination of hypergeometric functions ${}_2F_1$. Note that we defined $F_p$ as polynomials (integer values of $p$), however \eqref{oversphere} naturally extends to any real $p$. In particular, $p=1/2$ leads to an elliptic integral of the first kind, that gives slightly better upper bounds for $\sqrt{-\Delta}$, than tight $1$-frame (see \autoref{thmperimeter} and its proof in \autoref{sec:fractional}). Effectively we are using $1/2$-frame-like identity, that is available for the full orthogonal group.

\subsection{$p$-moments of the sums of squares of Gaussian random variables.}
From the probabilistic point of view, the Laplacian is closely connected to Brownian motion and Gaussian random variables. We tackle the general matrix $E$ in the integral above using Gaussian random variables.
Note that 
\begin{align*}
  \int_0^\infty e^{-r^2/2} r^{2p+d-1} dr=2^{p-1+d/2}\Gamma(p+d/2).
\end{align*}
Therefore \eqref{oversphere} can be rewritten as
\begin{align}\label{fpgaus}
  F_p(E^2)=
  \frac{\Gamma(d/2)}{2^{p}\Gamma(p+d/2)}\int_\Rd \frac{1}{(2\pi)^{d/2}}e^{-|x|^2/2} |Ex|^{2p}dx=\frac{1}{2^p (d/2)_p}\E\left(\sum_i e_i^2X_i^2\right)^p,
\end{align}
where $X_i$ are independent standard normal random variables. On the right we have a somewhat complicated $2p$-moment of these random variables.
Note that taking $e_1=\dots=e_k=1$ and $e_{k+1}=\dots=e_d=0$, then integrating resulting $k$-dimensional Gaussian leads to a different proof of \cite[Remark 5.2]{BE13} (constant in \eqref{PVframe}). 

To find the general formula we first establish two probabilistic results. Then we use abstract algebra of symmetric functions to deduce an easy to calculate formula for $F_p$. By $m_\lambda(a_1,\dots,a_n)$ we denote the monomial symmetric polynomial with exponent patern $\lambda$ (see e.g. \cite{Mac95}), that is a sum of all possible monomials with the same exponent pattern.

\begin{lemma}\label{lemgaus}
  Let $X_i,\dots,X_n$ be independent standard normal random variables, and $a_i\in \Real$. Then
  \begin{align}
 \E\left(\sum_i a_iX_i^2\right)^p 
 =\frac{p!}{2^p}\sum_{\lambda}\left(\prod_{k\in\lambda}\binom{2k}{k}\right)m_\lambda(a_1,\dots,a_n), \label{GM}
  \end{align}
  where we sum over all integer partitions $\lambda$ of $p$.
\end{lemma}
We do not know whether this is a new result, but we could not find it in any probability related book.
\begin{proof}
  
\begin{align*}
 \E\left(\sum_i a_iX_i^2\right)^p 
 &=\sum_{k_1+\dots+k_p=p} \binom{p}{k_1,\dots,k_d}\prod_{i=1}^n a_i^{k_i}\E X_i^{2k_i}
 \\&=
 \sum_{k_1+\dots+k_p=p} \binom{p}{k_1,\dots,k_d}\prod_{i=1}^n (2k_i-1)!! a_i^{k_i} 
\end{align*}
Note that multinomial coefficient, as well as the double factorial representing the Gaussian moment only depend on the multiset $\{k_i\}$. Therefore the sum can be rewritten as a sum over all partitions $\lambda$ of $p$, involving monomial symmetric polynomials $m_\lambda(a_1,\dots,a_n)$. Since $X_i$ are identically distributed, they have the same moments, and we can treat $a_i$ as variables of the polynomial. We get
\begin{align*}
 \E\left(\sum_i a_iX_i^2\right)^p 
 &= 
 {p!}\sum_{\lambda}\left(\prod_{k\in\lambda}\frac{(2k-1)!!}{k!}\right)m_\lambda(a_1,\dots,a_n) 
 \\&=
 \frac{p!}{2^p}\sum_{\lambda}\left(\prod_{k\in\lambda}\binom{2k}{k}\right)m_\lambda(a_1,\dots,a_n)
\end{align*}
\end{proof}
We also need an auxiliary result for double sequences of standard normal random variables. The following lemma relies on the relation between Gaussian random variables and the $\chi^2$ distribution. This time we need complete homogenous polynomials $h_p$, defined as the sum of all possible monomials of degree $p$. Again, the monograph \cite{Mac95} provides an excellent source of information about these and other kinds of symmetric polynomials.
\begin{lemma}\label{lemdouble}
  Let $X_i,\dots,X_n,Y_1,\dots,Y_n$ be independent standard normal random variables, and $a_i\in \Real$. Then
  \begin{align}
  \frac{p!}{2^p}
 \sum_{\lambda}\left(\prod_{k\in\lambda}\binom{2k}{k}\right)&m_\lambda(a_1,\dots,a_n,a_1,\dots,a_n)
 \\&=
 \E\left(\sum_i a_i(X_i^2+Y_i^2)\right)^p 
 =
 2^p p! h_p(a_1,\dots,a_n).
  \end{align}
\end{lemma}
Note that monomial symmetric polynomial $m_\lambda$ has $2n$ variables, but is evaluated at two copies of $a_1,\dots,a_n$. As a result we get a polynomial with more than one exponent pattern.

\begin{proof}
  The first equality follows directly from \autoref{lemgaus} with $2n$ standard normal random variables, with each factor $a_i$ repeated twice. To get the second formula, note that $X_i^2+Y_i^2$ has $\chi^2(2)$ distribution, which also equals exponential distribution with $\lambda=1/2$. Let $Z_i$ be exponential random variables with $\lambda=1$, then
  \begin{align}
 \E\left(\sum_i a_i(X_i^2+Y_i^2)\right)^p
 &=
 2^{p}\E\left(\sum_i a_i Z_i\right)^p =
 2^{p}\sum_{k_1+\dots+k_p=p} \binom{p}{k_1,\dots,k_p}\prod_{i=1}^n a_i^{k_i}\E Z_i^{k_i}
 \\&=
 2^{p}\sum_{k_1+\dots+k_p=p} \binom{p}{k_1,\dots,k_p}\prod_{i=1}^n k_i! a_i^{k_i}= 
 {2^p p!} h_p(a_1,\dots,a_n).
  \end{align}
  In the last equality we cancel all $k_i!$ terms with the multinomial coefficient, and we get all possible monomials of combined degree $p$, each with coefficient $p!$. Therefore we get complete homogeneous polynomial $h_p$ of degree $p$ in $n$ variables. 
\end{proof}

\subsection{Algebra of symmetric functions}
Consider the ring of symmetric functions over rational numbers, denoted by $\Lambda$ in \cite{Mac95}, defined on countably many variables $\{a_1,a_2,\dots\}$. Let $A_n=\{a_1,\dots,a_n\}$. For any partition $\lambda$ of $p$
we define a $\lambda$ indexed power sum basis element
\begin{align*}
  p_\lambda(A_n)=\prod_{k\in\lambda} (a_1^{k}+\dots+a_n^{k})=\prod_{k\in\lambda} p_{k}(A_n),
\end{align*}
where $p_k(A_n)$ is called a power sum polynomial of degree $k$.
Then $p_\lambda(A_n)$ generates $\Lambda$. In the previous section we used monomial symmetric polynomials, which also form a basis for the same ring. 

Note that any symmetric function has the same expansion in monomial basis (or power sum basis), regardless of the number of variables in $A_n$. In other words, $\Lambda$ can be treated as an abstract ring without ever specifying the variable set $A$ (alphabet). Furthermore, one can define abstract algebraic alphabet operations over rational numbers, and these operation reduce to intuitive alphabet operations over natural numbers. See \cite[Chapter 2]{Las03} for even broader context. In particular, for any symmetric function $S(A)$ we could try to define doubling operation
\begin{align*}
  S(2A)=s(a_1,\dots,a_n,a_1,\dots,a_n).
\end{align*}
Hence $S(2A)$ would be a symmetric function of $2n$ variables, evaluated on two copies of $A$. However it is not immediately clear how one could make this precise. Expanding $S$ in the power sum basis with rational coefficients $q_\lambda$ leads to an intuitive definition
\begin{align*}
  S(2A)&=\sum_\lambda q_\lambda p_\lambda(2A)=\sum_\lambda q_\lambda \prod_{k\in \lambda} p_k(2A)
  \\&=\sum_\lambda q_\lambda \prod_{k\in \lambda} 2p_k(A),
\end{align*}
since power sum polynomials are simply sums of powers of all variables. It is clear that using power sum basis we can define $S(qA)$ for any rational $q$. And for any natural number $q$ we can interpret this transformation as taking $qn$ variables evaluated on $q$ copies of $A_n$. It is also clear that $\theta_q: S(A)\to S(qA)$ is an automorphism of $\Lambda$ with inverse $\theta_{1/q}$. 

From now on, let
\begin{align}\label{Sgaus}
  S(A)=\E\left(\sum_i a_iX_i^2\right)^p 
\end{align}
\autoref{lemgaus} implies that $S(A)$ a symmetric function over rational numbers, expanded in monomial symmetric basis $m_\lambda$. But \autoref{lemdouble} gives
\begin{align*}
  S(2A)=2^pp!h_p(A).
\end{align*}
Note that direct proof of this fact would be tedious, since monomial symmetric polynomials $m_\lambda$ do not interact in a simple way with automorphism $\theta$. It is also not trivial to transform from monomial to power sum basis. 

Invertibility of $\theta$ gives
\begin{align}\label{stoh}
  S(A)=2^p p! h_p\left(\frac{1}{2}A\right).
\end{align}
This provides a very concise formula for $S(A)$. Unfortunately, action of $\theta$ on complete homogeneous polynomials is not simple. Hence we will rewrite this expression using power sum basis.

\subsection{Cycle index in power sum basis}

For any finite group $G$ of permutations we can define a so-called cycle index $Z(G,c_1,\dots,c_p)$ in the following way. Consider a sequence of dummy variables $c_1,\dots,c_p$, and  use $c_i$ to denote cycle of length $i$. Every permutation can now be represented using a monomial of degree equal to the number of cycles in this permutation. The cycle index is an average of those monomials over all elements of the group. Note that $c_i$ corresponds to cycles of length $i$, hence replacing $c_i$ with power sum polynomial $p_i(A)$ leads to a homogeneous symmetric polynomial in alphabet $A$ (variables $a_1,\dots, a_n$). 

We are interested in the cycle index of the symmetric group $S_p$. We get a polynomial, whose coefficients count number of permutations with the cycle lengths given by $\lambda$ (denote these coefficients by $q_\lambda$). We have
\begin{align}\label{Zs}
  Z\left(S_p,a_1+\dots+a_p,a_1^2+\dots+a_n^2,\dots,a_1^p+\dots+a_n^p\right)=\frac{1}{p!}\sum_\lambda q_\lambda p_\lambda(A).
\end{align}
Simple counting argument shows that if $\lambda$ has $j_k$ cycles of length $k$ then
\begin{align}\label{sl}
  q_\lambda=\frac{p!}{\prod_{k=1}^p k^{j_k} j_k!}.
\end{align}

However, P\'olya's enumeration theorem (generalized Burnside's lemma, see \cite{R27,P37} or any modern enumerative combinatorics book) states that this cycle index is a generating function of colorings of $p$ objects with $n$ colors (monomial $\prod_{i=1}^n a_i^{k_i}$ encodes $k_i$ elements of color $i$). But the colorings are considered the same if there is a permutation of objects that maps one coloring to the other. This last restriction clearly implies that there is exactly one coloring per monomial. Hence
\begin{align}\label{Zh}
  Z\left(S_p,a_1+\dots+a_p,a_1^2+\dots+a_n^2,\dots,a_1^p+\dots+a_n^p\right)= h_p(A).
\end{align}
Therefore $s(A)$ defined by \eqref{Sgaus} satisfies
\begin{align}
  S(A)&=2^pp!h_p\left( \frac{1}{2}A \right)=4^p\sum_\lambda q_\lambda p_\lambda\left( \frac{1}{2}A \right)=2^pp!\sum_\lambda q_\lambda \prod_{k\in \lambda} \frac{1}{2} p_k(A)\nonumber
  \\&=2^pp! 
  Z\left(S_p,\frac{a_1+\dots+a_n}{2},\frac{a_1^2+\dots+a_n^2}{2},\dots,\frac{a_1^p+\dots+a_n^p}{2}\right)\label{sZ}
\end{align}
To find $S(A)$ one needs to replace $c_i$ in the abstract cycle index of $S_p$ with $p_i(A)/2$.

\subsection{Schatten norms and formulas for polynomials $F_p$.}
Let $T$ be an arbitrary rectangular matrix, and $s(T)$ the multiset of its singular values. The Schatten norm of $T$ of order $2n$ equals
\begin{align*}
  \|T\|_{2n}^{2n} = \sum_{\sigma\in s^2(T)} \sigma^{2n}= p_n(s^2(T)),
\end{align*}
where $p_n$ is the power sum polynomial of order $n$. This description fits perfectly into the framework of the power sum basis for symmetric polynomials discussed in the previous section. However, the same norm can also be calculated using matrix trace
\begin{align*}
  \|T\|_{2n}^{2n} = \tr\left((T^\dagger T)^n\right) = \tr\left( (TT^\dagger)^n\right),
\end{align*}
where $T^\dagger$ denotes the transposed matrix. This formula allows for calculating Schatten norms without finding singular values of $T$. One can also choose the order of multiplication that gives smaller matrices. This characterization can be useful when matrix $T$ involves further unknowns, in which case finding singular values would be hard.

Now we return to polynomials $F_p(s^2(T))$. \autoref{lemgaus} and \eqref{fpgaus} provide a combinatorial formula:
\begin{lemma}\label{lemmafpm}
  For any (possibly rectangular) matrix $T$
  \begin{align*}
    F_p(s^2(T))= 
    \frac{p!}{4^p (d/2)_p}\sum_{\lambda}\left(\prod_{k_i\in\lambda}\binom{2k_i}{k_i}\right)m_\lambda(s^2(T)).
  \end{align*}
\end{lemma}
This result is easy to apply when $p$ is small. In particular, when $p=2$ the generalized tight $2$-frame identity reads
  \begin{align}
    \frac{1}{|G|}\sum_U |TUx|^4  &= \frac{1}{d(d+2)}\left(3\sum_{i=1}^d s_i^4 + 2\sum_{i<j} s_i^2s_j^2  \right)|x|^4\nonumber
    \\&=
    \frac{\|T\|_2^4+2\|T\|_4^4}{d^2+2d}|x|^4,\label{schatten4}
  \end{align}
  where $s_i$ are singular values of $T$. Note that the second formula involving Schatten norms allows us to find the value of $F_p$ without finding singular values of $T$. 
  When $T$ is an identity matrix, the Schatten $p$-norm equals $\sqrt[p]{d}$. Hence the $T$ dependent term reduces to $1$, as expected. For $p=3$ we get  
  \begin{align*}
  \frac{1}{|G|}\sum_U |TUx|^6  &= \frac{1}{ d(d+2)(d+4)}\left(15\sum_{i=1}^d s_i^6 + 9\sum_{i<j} (s_i^4s_j^2+s_j^4s_i^2) + 6\sum_{i<j<k} s_i^2 s_j^2 s_k^2 \right)|x|^6
    \\&=
    \frac{\|T\|_2^6+6\|T\|_4^4\|T\|_2^2+8\|T\|_6^6}{d^3+6d^2+8d}|x|^6,
\end{align*} 
Alert reader may notice that coefficients in the numerator are the same as in the denominator. This is only true for $p=2$ and $p=3$, since the number of partitions of those integers equals $p$. For $p=4$, we already have five partitions, including two of length two, $4=3+1=2+2$, and it is no longer clear which one should be used. To obtain Schatten-type formula for $F_p$ we must use the symmetric polynomial approach developed in the previous section. 

For arbitrary $d$ and arbitrary $p$ the following Schatten-type formula seems most convenient, especially when finding singular values is not practical.
\begin{theorem}\label{thmcycle}
  For any (possibly rectangular) matrix $T$
  \begin{align}
    F_p(s^2(T))&=
    \frac{p!}{(d/2)_p} Z\left(S_p,\frac{\|T\|_2^2}{2},\frac{\|T\|_4^2}{2},\dots,\frac{\|T\|_{2p}^{2p}}{2}\right)
    \\&=
    \frac{p!}{(d/2)_p}\sum_{j_1+2j_2+\dots+pj_p=p} \prod_{k=1}^p \frac{\|T\|_{2k}^{2kj_k}}{(2k)^{j_k}j_k!}.\label{fpcycles}
  \end{align}
  where $Z(S_p,c_1,\dots,c_p)$ is the cycle index for the symmetric group $S_p$.
\end{theorem}
\begin{proof}
  The first formula follows from \eqref{sZ} combined with \eqref{Sgaus} and \eqref{fpgaus}. To get the second formula, start with \eqref{Zs}, decompose each $\lambda$ into $j_k$ cycles of length $k$, and apply formula \eqref{sl} for $q_\lambda$.
\end{proof}

A very restricted choice of partitions in dimension 2 (any monomial symmetric polynomial involving partition with 3 or more pieces equals $0$) allows for a simpler singular value based formula for arbitrary value of $p$
\begin{align*}
  \frac{1}{|G|}\sum_U |TUx|^{2p}  &= 
  \frac{1}{4^p}\sum_{k=0}^p\binom{2k}{k}\binom{2p-2k}{p-k} s_1^ks_2^{p-k}.
\end{align*}
Unfortunately, it seems that the simplest form of this series is a hypergeometric function in $s_1/s_2$. Even taking $s_1=s_2=1$ (identity matrix), leads to a nontrivial combinatorial identity
\begin{align*}
  \sum_{k=0}^p\binom{2k}{k}\binom{2p-2k}{p-k}=4^p.
\end{align*}

\subsection{A note about generalized $\chi^2$ distributions}\label{chisquared}
Let $X_1,\dots,X_n$ be independent, centered Gaussian random variables with variances $\sigma_i$. \autoref{lemgaus} provides a formula for the $p$-moment of $\sum_i X_i^2$, a special case of the generalized $\chi^2$ distribution. 

On the other hand, \autoref{lemdouble} gives a much simpler formula for a related $\chi^2$, with each Gaussian repeated twice. One can also think about this case as a sum of the squared magnitudes of complex gaussian random variables. The relation between this distribution and exponential random variables was the key to a simplicity of \autoref{lemdouble}. 

Symmetric polynomial manipulations, in particular formula \eqref{stoh}, allow us to find a concise formula for $p$-moments of the real case $\sum_i X_i^2$ in terms of the complex (double) $\chi^2$. Combining this theoretical result with cycle index decomposition \eqref{Zh} leads to another way of calculating the same moments. Finally, symmetric polynomial manipulations, including $\theta$ authomorphism, are easy to perform using SAGE computer algebra system, providing a convenient way of finding these moments (see section below). 
\subsection{SAGE code}\label{secsage}
In this section we explore possibility to quickly evaluate polynomials $F_p$ using SAGE Computer Algebra System \cite{Sage}. We used Sage Cell Webserver \cite{sagecell} to perform all calculations and to provide direct links to these calculation.

The cycle index of the symmetric group is a well known polynomial implemented in SAGE via GAP library. In particular, one can obtain the cycle index (with powersums as variables) using the
\href{https://sagecell.sagemath.org/?z=eJzLU7BVMOHlCq7MzU0tKcpMdi_KLy3QyNPUS65MzkmNz8xLSa3Q0NQryUgtSdQw1DfSBACOWw_t&lang=sage}{following SAGE code}

\lstset{
	language=Python,
	showstringspaces=false,
        breaklines=true,
        frame=tb
	xleftmargin=\parindent,
        basicstyle=\scriptsize\ttfamily,
	keywordstyle=\bfseries\color{green!40!black},
        identifierstyle=\color{blue!80!black},
	commentstyle=\itshape\color{purple!80!black},
	stringstyle=\itshape\color{orange!90!black},	
	escapeinside={*@}{@*}
}
\begin{lstlisting}
n = 4 # frame order
cycle_index = SymmetricGroup(n).cycle_index()
Z = cycle_index.theta(1/2) # apply theta automorphism
print Z
\end{lstlisting}
The output
\begin{lstlisting}
1/384*p[1,1,1,1] + 1/32*p[2,1,1] + 1/32*p[2,2] + 1/12*p[3,1] + 1/8*p[4]  
\end{lstlisting}
is the expansion of the cycle index from \autoref{thmcycle} in terms of powersum basis, after the application of the $\theta$ automorphism. E.g. $p[2,1,1]$ corresponds to $\|T\|_{4}^{4}(\|T\|_{2}^{2})^2$ in the formula \eqref{fpcycles} for $F_p(s^2(T))$.

Another way of obtaining the cycle index is to explicitly convert complete homogeneous polynomial to powersum basis using \eqref{Zh}. In SAGE (Symmetrica package), \href{https://sagecell.sagemath.org/?z=eJzLU7BVMOHlCq7MzU0tKcpMdivNSy7JzM8r1ggM1NTLzMtKTS6JL87ILyrJSMxLKdaIVi9Q11HPUI_V5OUq0MiIzovV1CvJSC1J1DDUN9IEAAMqGR4=&lang=sage}{the following input} produces the same cycle index $Z$ as above
\begin{lstlisting}
n = 4 # frame order
SymmetricFunctions(QQ).inject_shorthands(["p","h"])
Z = p(h[n]).theta(1/2) # convert h to p, and apply theta authomorphism
\end{lstlisting}

Abstract power sum polynomials in Sage can be expanded in any number of variables (SAGE Symmetrica library). Using cycle index $Z$ we just calculated, we obtain singular value based expansion for $F_p$ given by \autoref{lemmafpm} (or \eqref{fpcycles} with Schatten norms expressed using singular values) by \href{https://sagecell.sagemath.org/?z=eJzLU7BVSAFiE16utMTkkvyizMQcjTxN_aLM4sy89HiEWIq-kU6eplZwZW5uaklRZrJ7UX5pAVClXnJlck5qfGZeSmqFhqZeSUZqSaKGob6Rpl5qRUFiXopGiiYAeG8hTg==&lang=sage}{executing}
\begin{lstlisting}
d = 4 # dimension
Fp = factorial(n)/rising_factorial(d/2,n)*Z.expand(d) # expand in dimension d
\end{lstlisting}

Let us also point out that Molien series that counts invariant polynomials for permutation groups is also implemented in Sage. Unfortunately, the series uses the representation of the group as permutation matrices, hence in a too-high dimension. However, one can implement Molien's series using Molien's theorem and Coxeter groups acting as reflections (with the help of MPFI, GAP, MPFR, ginac, GMP and Maxima libraries). This allows one to quickly check if a given group of isometries admits higher order frames.  In particular, to check uniqueness of the invariants for the group of rotations (no reflections) of a tetrahedron 
\href{https://sagecell.sagemath.org/?z=eJxlj8EKwjAQRO8F_yH0lEiNFr3mIB5y6heUUkLdhkialE0q7d-bVkHQ2zK7s_NGipufIQJK9NNI6_yaF-emMMNoYQAXVTTeiRyht9Ctc852mRSVimjmt0dyDS5Q9reoNek9Ek2MI5JbEyJlxPRE82G7oozfIWlClE0yI4xBSN5595i06pa2syqENskI4Y3yhJSzyypvDTgRpoHWFhzFX1OCOYa1g-kXWh7mPX6iNiBcgda0hh0l93gH_H7lUS3WI52LU3FJ6gtfmGBx&lang=sage}{use the following code}
\begin{lstlisting}
G=CoxeterGroup(["A",3],implementation="reflection") # tetrahedral group in R3
G=MatrixGroup(G.gens()) # as matrix group
G=MatrixGroup([g for g in G.list() if g.matrix().det()==1]) # rotations only
reps=G.conjugacy_class_representatives()
# Molien's theorem
M=sum([len(r.conjugacy_class())/simplify(1-x*r).det() for r in reps])/G.order()
M.taylor(x,0,4) # expand as series
\end{lstlisting}
This outputs
\begin{lstlisting}
2*t^4 + t^3 + t^2 + 1
\end{lstlisting}
We have exactly one invariant for degrees $0$, $2$ and $3$, but two invariants of degree $4$. 

Finally, note that we can also find all fundamental invariants in Sage and check their degrees (using PARI, MPFI, Singular, GAP, FLINT, MPFR, ginac, GMP and NTL libraries). In particular for the rotation group of the icosahedron \href{https://sagecell.sagemath.org/?z=eJxljT0LgzAURXfB_xCcEpBAkY6ZhKaLk6NICfoMD2ISnmmx_77pB-3Q7X2ce65WbdghAWkK18iH6lzVzVjjGh2s4JNJGLyqCBYH03OuRFl0vepMItz7aCbg7X1yIYUVpxOCm_lR1E3dZE5_sLdbSwt-4-LvMXQ9t4ItgZhl6JmWDrfEBcOFWbm-SC7kDPmm1GHMAvQ3Q2h82pSW3-WSC4BMCpRrymKIOWMJgL_l8Sn_JceyeAAmzFQl&lang=sage}{we have}
\begin{lstlisting}
G=CoxeterGroup(["H",3],implementation="reflection") # icosahedral group in R3
MS=MatrixSpace(CyclotomicField(5),3,3) # icosahedron uses 5th root of unity
G=MatrixGroup([MS(g) for g in G.gens()]) # change the base ring for generators
G=MatrixGroup([g for g in G.list() if g.matrix().det()==1]) # rotations only
invariants=G.invariant_generators() 
print [p.degree() for p in invariants]
\end{lstlisting}
And we get the following degrees
\begin{lstlisting}
[2, 6, 10, 15]
\end{lstlisting}
Note however, that the representation using reflections (generators satisfy $M^2=\Id$) does not give the subgroup of the orthogonal group (generators should satisfy $MM^\dagger=\Id$). Instead, one gets a group that acts on a sheared regular polytope. Therefore the invariants will not be the same as the invariants we use in this paper, but the degrees coincide. We see that the group of rotations of icosahedron admits $2$-frames, but not $3$-frames, just like its full symmetry group.

\section{(Non-tight) $p$-frames for any irreducible groups of symmetries}\label{secnontight}
Let $G$ be an irreducible group of symmetries. This group may or may not admit higher order frames, but for any matrix $T$ we have the tight $1$-frame identity
\begin{align}\label{Tframeid}
  \frac{1}{|G|}\sum_U |TU x|^2=\frac{1}{d}\|T\|_2^2 |x|^2.
\end{align}
We also have the inner product version of this identity, \autoref{matrixframe} with $p=1$.
Note that if $T$ has just one row, this reduces to classical tight frame identity for vectors. Note also that with many rows we are simply adding one tight frame identity per row. Upper bounds for eigenvalues obtained in \cite{LS11} rely on this identity.

For a given multiplier $f$, a special case of its transformed version defined in \eqref{newmult} reads
\begin{align*}
  F[\Phi_1,0,T](t) = \frac{1}{d}\tr \Phi_1(t T^{-\dagger}T^{-1}).
\end{align*}
Recall that the trace of a function of a matrix means the trace of the spectrally defined function of the matrix (see the paragraph above \eqref{newmult}).

\begin{theorem}\label{nontight}
  For any irreducible group of symmetries $G$ acting on $\Rd$, any convex function $\Phi$ and any concave function $\Psi$
  \begin{align*}
    \frac{1}{|G|}\sum_U \Psi(\Phi(|T^{-1}U x|^2))\le \Psi\circ F[\Phi,0,T](|\xi|^2),\\
    \frac{1}{|G|}\sum_U \Phi(\Psi(|T^{-1}U x|^2))\ge \Phi\circ F[\Psi,0,T](|\xi|^2).
  \end{align*}
  Furthermore, we get equalities if and only if $T$ is a multiple of an orthogonal matrix, or both $\Psi$ and $\Phi$ are linear.
\end{theorem}
As an immediate corollary we get a bound for the constants $F_p$ in tight $p$-frame identities.
\begin{corollary}\label{cornontight}
  Polynomials $F_p$ used in the tight $p$-frame identities (\autoref{theoremframes} and \autoref{corframes}) satisfy
  \begin{align*}
    \frac1{d^p}{\|T\|_{2}^{2p}}\le F_{p}(s^2(T))\le \frac1d{\|T\|_{2p}^{2p}}.
  \end{align*}
  Hence, orbits of any irreducible group $G$ form (not necessarily tight) $p$-frames.
\end{corollary}
\begin{proof}[Proof of \autoref{nontight}]
  Note that we can assume that $|x|=1$ by considering $|x|$ a part of the matrix $T$. First consider the case $\Phi\ge 0$.

  Write the singular value decomposition of $T^{-1}=WSV$ with diagonal matrix $S$ and orthogonal matrices $W$ and $V$. If $|v|=1$, then the inner product $\langle S^2 v,v\rangle$ can be interpreted as a sum over a probability measure given by $v_i^2$. Therefore Jensen inequality gives
  \begin{align*}
    \Phi(|T^{-1}Ux|^2)&=\Phi(\langle S^2\, VUx,VUx\rangle)\le \langle \Phi(S^2)\, VUx,VUx\rangle
    \\&=\tr (\Phi(S^2)VUxx^\dagger U^\dagger V^\dagger)
    =\tr(V^\dagger \Phi(S^2) VP_{Ux})=\langle V^\dagger \Phi(S^2) V,P_{Ux}\rangle_{HS}
  \end{align*}
  Note that we get equality if and only if $S$ is a multiple of identity. Note also that we can replace $\Phi$ with $\Phi+l$ where $l$ is linear (not necessarily nonnegative). Concavity of $\Psi$, with group averaging as a convex combination, gives
  \begin{align*}
    \frac{1}{|G|}\sum_U &\Psi(\Phi(|T^{-1}Ux|^{2}))\le \Psi\left(\frac{1}{|G|}\sum_U \Phi(|T^{-1}Ux|^{2})\right)\le \Psi\left(\frac{1}{|G|} \sum_U \langle V^\dagger \Phi(S^2) V,P_{Ux}\rangle_{HS}\right)
    \\&=\Psi\left(\frac{1}{d}\|V^\dagger \Phi(S^2)V\|_{1}\right)
    \le \Psi\left(\frac{1}{d}\|\Phi(T^{-\dagger}T^{-1})\|_{1}\right)=\Psi\left(\frac{1}{d}\tr(\Phi(T^{-\dagger}T^{-1}))\right)
  \end{align*}
  using \eqref{matrixframe} with $p=1$. We also get the equality statement for the upper bound if and only if $T$ is a multiple of an orthogonal matrix, or both $\Psi$ and $\Phi$ are linear. Note that we need $\Phi\ge0$, so that the Schatten $1$-norm equals the trace of a matrix. However, as above, we can also handle the case $\Phi+l$ by applying tight frame identity separately to $\Phi$ and $l$. The latter gives equalities in the above calculation.

Note that for general convex $\Phi$ there exists linear function $l$ (any supporting line or affine minorant), such that $\Phi\ge l$. Then $\Phi_1=\Phi-l$ is nonnegative and convex. Now repeat the proof with $\Phi=\Phi_1+l$.  Hence the theorem holds for arbitrary convex functions $\Phi$.

  Finally, all inequalities in the above proof come from convexity of $\Phi$ and concavity of $\Psi$. We can replace $\Phi$ and $\Psi$ with $-\Phi$ and $-\Psi$ to get the opposite inequalities, proving the lower bound in the theorem.
\end{proof}

\begin{remark}
In the above proof we used convexity/concavity to turn $\Phi$-frame type expression into a $1$-frame. We needed to do this to obtain valid $1$-frame identity for $G$. However, straightforward modification of the above proof leads to non-tight frame identities based on $p$-frames whenever these are admissible for a given symmetry group $G$.
\end{remark}

\autoref{nontight} provides a framework for handling almost arbitrary multipliers. Any function possessing left and right derivatives, such that those derivatives are of bounded variation on any closed bounded subinterval of $(0,\infty)$, can be represented using a difference of two convex functions (class of DC functions, see \cite{H59} or \cite{BB11}). Hence any function $f=\Phi_1-\Phi_2$, where $\Phi_i$ are convex, can be handled by \autoref{nontight}. In particular any $f\in C^{1,1}$ can be written as a difference of nonnegative convex functions.

\begin{proof}[Proof of \autoref{thmgeneralmult}]
  Start with \eqref{averagesum} with $f=\Phi_1-\Phi_2$
\begin{align}
  \lambda_1+\dots+\lambda_n\Big|_{T(\Omega)}&\le
  \frac{1}{|G|}\sum_{U\in G}\sum_{i=1}^n\int_{\Rd}(\Phi_1-\Phi_2)(|T^{-\dagger}U^\dagger \xi|^2)|\widehat u_i|^2d\xi
\end{align}
Now apply \autoref{nontight}, first with $\Psi(t)=t$ and $\Phi=\Phi_1$, then with $\Psi=-\Phi_2$ and $\Phi(t)=t$. On the right we get a sum of quadratic forms for the multiplier from \autoref{thmgeneralmult}. Choose $u_i$ to be the eigenfunctions for that multiplier to get the result. Similarly we get the second part of the theorem.
\end{proof}

\section{Higher moments of mass}\label{secmoments}
As the first application of generalized $p$-frames, we find a relation between higher order moments mass of a highly symmetric domain $\Omega$ and its linearly transformed image. This section generalizes \cite[Lemma 9]{LS11}, where the second moment (the moment of inertia) is treated using classical tight frames.

In \autoref{mainclamped} we defined the $2p$-moment of mass as 
\begin{align*}
  I_{2p}(\Omega)=\int_\Omega |x|^{2p} dx.
\end{align*}
Note that rescaling the domain with a fixed scaling factor $c$, scales $I_{2p}$ by $c^{2p+d}$. In particular volume ($p=0$) scales like $c^d$. Hence $V^{1+2p/d}/I_{2p}$ is a scale invariant quantity. Moreover

\begin{lemma}\label{lemmoments}
  If $\Omega$ admits $p$-frames, then
  \begin{align*}
    F_p(s^2(T)) = \frac{I_{2p}(T(\Omega))}{I_{2p}(\Omega)} \frac{V(\Omega)^{1+4p/d}}{V(T^{-1}(\Omega))^{2p/d}V(T(\Omega))^{1+2p/d}}
  \end{align*}
  Furthermore, for any $\Omega$ with irreducible isometry group
  \begin{align*}
    \frac{1}{d^p} \|T\|_{HS}^{2p}\le\frac{I_{2p}(T(\Omega))}{I_{2p}(\Omega)} \frac{V(\Omega)^{1+4p/d}}{V(T^{-1}(\Omega))^{2p/d}V(T(\Omega))^{1+2p/d}} \le \frac{1}{d}\|T\|_{2p}^{2p}.
  \end{align*}
\end{lemma}

\begin{proof}
Suppose $\Omega$ admits tight $p$-frames. Let $T$ be an arbitrary linear transformation. Then
\begin{align*}
  I_{2p}(T(\Omega))=\int_{T(\Omega)} |x|^{2p} dx = |T| \int_{\Omega} |T x|^{2p}  dx
\end{align*}
But $\Omega$ is invariant under its isometry group $G$, hence
\begin{align*}
  I_{2p}(T(\Omega)) = |T| \frac{1}{|G|}\sum_U \int_{\Omega} |T Ux|^{2p}  dx=|T| F_p(s^2(T)) I_{2p}(\Omega),
\end{align*}
by \autoref{corframes}. Similarly, if the group of symmetries of $\Omega$ is just irreducible, \autoref{nontight} implies
\begin{align*}
  \frac{1}{d^p} \|T\|_{HS}^{2p}\le\frac{I_{2p}(T(\Omega))}{I_{2p}(\Omega)} |T|^{-1} \le \frac{1}{d}\|T\|_{2p}^{2p}.
\end{align*}

As in the proof of \cite[Lemma 9]{LS11} we finish the proof by noting that
\begin{align*}
  |T|=\frac{V(T^{-1}(\Omega))^{2p/d}V(T(\Omega))^{1+2p/d}}{V(\Omega)^{1+4p/d}}.
\end{align*}
\end{proof}

Using Schatten norms based formula for $F_p$ we get
\begin{lemma}
  In two dimensions
  \begin{align*}
    F_p(s^2(T^{-1}))|T|^{2p}=F_p(s^2(T))
  \end{align*}
\end{lemma}
\begin{proof}
  Note that
  \begin{align*}
    \|T^{-1}\|_{k}^{k}=s_1^{-k}+s_2^{-k}=\frac{s_1^k+s_2^k}{(s_1s_2)^k}=\frac{\|T\|_k^k}{|T|^k}.
  \end{align*}
  Combine this Schatten norm property and \autoref{thmcycle} to get the result.
\end{proof}
This allows us to compare the higher moments of $T(\Omega)$ and $T^{-1}(\Omega)$ in dimension two.
\begin{lemma}\label{momentstwo}
  Let $A$ be the area of a two dimensional domain $\Omega$. Then
  \begin{align*}
    \left.\frac{A^{1+p}}{I_{2p}}\right|_{T(\Omega)}=
    \left.\frac{A^{1+p}}{I_{2p}}\right|_{T^{-1}(\Omega)}
  \end{align*}
\end{lemma}
\begin{proof}
  \begin{align*}
    I_{2p}(T(\Omega))=|T|F_p(s^2(T))I_{2p}(\Omega)=|T|^{2+2p}|T^{-1}|F_p(s^2(T^{-1}))I_{2p}(\Omega)=|T|^{2+2p}I_{2p}(T^{-1}(\Omega))
  \end{align*}
  Now divide by $|T|^{1+p}$ and note that $|T|^{-1}=|T^{-1}|$.
\end{proof}
\section{Bi-Laplacian and convex multipliers}\label{secbilap}

\subsection{Plate problem}
In this section we apply the theory of $p$-matrix frames to the Fourier multipliers of the form $\Phi(|\xi|^2)$ with convex $\Phi$. Consider a quadratic form
\begin{align*}
  Q_A(u,v)=\int_{\Rd} \Phi(|\xi|^2) \widehat u(\xi)\overline{\widehat v(\xi)}\,d\xi. 
\end{align*}
We get a weak formulation of the underlying operator $A$, with discrete spectrum for any domain $\Omega$ (due to Dirichlet boundary condition). As described in \autoref{qforms}, the smallest eigenvalue can be obtained by minimizing Rayleigh quotient, and the sums via finite subspace minimization. Recall \eqref{averagesum}, but with convex multiplier $\Phi$

\begin{align}
  \lambda_1+\dots+\lambda_n\Big|_{T(\Omega)}&\le
  \frac{1}{|G|}\sum_{U\in G}\sum_{i=1}^n\int_{\Rd}\Phi(|T^{-\dagger}U^\dagger \xi|^2)|\widehat u_i|^2d\xi\nonumber
  \\&=
  \sum_{i=1}^n\int_{\Rd}\left(\frac{1}{|G|}\sum_{U\in G}\Phi(|T^{-\dagger}U^\dagger \xi|^2)\right)|\widehat u_i|^2d\xi\label{plateaverage}
\end{align}

Bi-Laplacian is perhaps the most import example of such multipliers. In \autoref{mainclamped} we described the differential equation leading to a vibrating plate problem. In terms of Fourier multipliers and quadratic forms we have
\begin{align*}
  Q_A(u,u)=\int_{\Rd} (|\xi|^4+\tau |\xi|^2) |\widehat u|^2\,d\xi.
\end{align*}
Hence bi-Laplacian can be weakly defined on $H^2_0(\Omega)$. Since this is a subspace of $H^1_0(\Omega)$, we clearly get compact embedding into $L^2$, and discrete spectrum.

The function $\Phi(t)=t^2+\tau t$ is convex, even for negative values of $\tau$. Hence we can apply \autoref{thmgeneralmult} with $\Phi_1=\Phi$ and $\Phi_2=0$. This leads to
\begin{align*}
  \lambda_1+\dots+\lambda_n\Big|_{|\xi|^4+\tau |\xi|^2,T(\Omega)}\le \lambda_1+\dots+\lambda_n\Big|_{(|\xi|^4\|T\|_4^4+\tau|\xi|^2\|T\|^{2})/d,\Omega}
\end{align*}
Due to almost homogeneous nature of the multiplier, we can take $\|T\|_4^4/d$ out of the quadratic form. Furthermore we can rescale $\tau$ so that we have no dependence on $T$ on $\Omega$ and we obtain \autoref{plateweaker} from \autoref{thmplate}. Note that the multiplier is not homogeneous, hence we must have some dependence of the multipliers in \autoref{plateweaker} on transformation $T$. We chose to keep the operator on the symmetric domain $\Omega$ as simple as possible.

Now we can use tight $2$-frames to improve the result for domains $\Omega$ that admit tight $2$-frames. We skip \autoref{thmgeneralmult} in favor of a direct approach based on \eqref{plateaverage}. We have
\begin{align*}
  \lambda_1+\dots+\lambda_n\Big|_{T(\Omega)}&\le
  \sum_{i=1}^n\int_{\Rd}\left(\frac{1}{|G|}\sum_{U\in G} (|T^{-\dagger}U^\dagger \xi|^4+\tau|T^{-\dagger}U^\dagger \xi|^2 )\right)|\widehat u_i|^2d\xi
  \\&=
  \sum_{i=1}^n\int_{\Rd}\left( F_2(s^2(T^{-1})) |\xi|^4+\tau \frac{\|T^{-1}\|_2^2}{d}|\xi|^2 \right)|\widehat u_i|^2d\xi
\end{align*}
using \autoref{corframes}. Define $D(T^{-1})=d\,F_2(s^2(T^{-1}))$ (or use \autoref{schatten4} directly) to get \autoref{platestronger} from \autoref{thmplate}. Note that we need to choose $u_i$ to be the eigenfunctions for the multiplier we have on the right of the inequality.

Finally, $D(T^{-1})\le \|T^{-1}\|_4^4/d$ is a special case of \autoref{cornontight}.

\subsection{Higher order Laplacians}\label{secpolylap}
When $f(t)$ is a polynomial, then the operator $f(-\Delta)$ is a higher order differential operator. We can estimate its eigenvalues using \autoref{thmgeneralmult} and, if needed, the convexification procedure described just under that theorem. We can however improve the bounds when the symmetry group of $\Omega$ admits tight $2$-frames. Assuming $f''(t)\ge -c$, we define
\begin{align*}
  \Phi_1(t)=f(t)+ct^2,\\
  \Phi_2(t)=ct^2.
\end{align*}
We handle $\Phi_1$ as in \autoref{thmgeneralmult}. However, we use tight $2$-frame identity on $-\Phi_2$ (as in the plate problem). Note that this is possible even for negative terms, since $2$-frames provide exact Rayleigh quotient evaluation for the multiplier $|\xi|^4$. As a consequence we obtain a stronger version of \autoref{thmgeneralmult}
  \begin{align*}
    \lambda_1+\dots+\lambda_n \Big|_{f,T(\Omega)}\le \lambda_1+\dots+\lambda_n\Big|_{G[\Phi_1,T],\Omega}.
  \end{align*}
with
  \begin{align}
    G[\Phi_1,T](t) = \frac{1}{d}\tr \Phi_1(t T^{-1}T^{-\dagger})-c F_2(s^2(T^{-1}))t^2.
\end{align}

  We can also completely forgo the use of \autoref{thmgeneralmult} if $\Omega$ admits tight $p$-frames for $2p$ larger or equal to the degree of the polynomial $f$. Then each monomial in $f$ can be exactly averaged using appropriate higher order tight frame, leading to further improvements. In particular, poly-Laplacian $\Delta^p$ with $p>1$ corresponds to $f(t)=t^p$. If $\Omega$ admits tight $p$-frames (in particular for disks), then the eigenvalues of the linearly transformed domain (e.g. ellipse) can be estimated by $F_p(s^2(T))$ times the eigenvalues on $\Omega$. Furthermore, $F_p(s^2(T))$ can be rewritten as the $2p$-moment of mass using \autoref{lemmoments}, leading to an  analog of \autoref{plateiso} with $I_4$ replaced by $I_{2p}$, and properly adjusted powers of the volume. 

\subsection{Buckling problem}\label{secbuckling}

Finally, we tackle a related plate buckling problem. We want to bound the principal eigenvalue $\Lambda$ in 
\begin{align*}
  \Delta^2 u +\Lambda \Delta u = 0\text{ in }\Omega,\\
  u=\frac{\partial u}{\partial \nu} = 0,\text{ on }\partial\Omega.
\end{align*}
The eigenvalue $\Lambda$ corresponds to the critical compression level that forces a plate to buckle. See e.g. Bramble and Payne \cite{BP63}, Ashbaugh and Laugesen \cite{AL96}, or Henrot \cite{He06} for known results and history of the problem.

The approach described above cannot be directly used to find a buckling eigenvalue estimate, due to the more complicated form of the Rayleigh quotient. The lowest buckling eigenvalue equals
\begin{align*}
  \Lambda(\Omega)=\inf_{u\in H_0^2(\Omega)} \frac{\int_{\Rd} |\xi|^4 |\widehat u|^2 d\xi}{\int_{\Rd} |\xi|^2 |\widehat u|^2 d\xi}.
\end{align*}
Therefore the numerator is the same as for bi-Laplace eigenvalues, while the denominator is the quadratic form of the Laplacian.
In order to apply our method we assume that $u$ is the eigenfunction for $\Lambda(\Omega)$ and we rewrite a variational upper bound on $T(\Omega)$ as
\begin{align*}
  \Lambda(T(\Omega))\int_{\Rd}|T^{-\dagger}U^\dagger \xi|^2|\widehat u|^2d\xi \le \int_{\Rd}|T^{-\dagger}U^\dagger \xi|^4|\widehat u|^2d\xi.
\end{align*}
We average this inequality (each side) over the group $G$, as before. 
\begin{align}\label{buckineq}
  \Lambda(T(\Omega))\int_{\Rd}\left(\frac{1}{G}\sum_{U\in G}|T^{-\dagger}U^\dagger \xi|^2\right)|\widehat u|^2d\xi \le \int_{\Rd}\left(\frac{1}{G}\sum_{U\in G}|T^{-\dagger}U^\dagger \xi|^4\right)|\widehat u|^2d\xi.
\end{align}
On the left, we can use the tight frame property for $G$. On the right, we either use $2$-frame identity if the group $G$ allows for it, or a bound from \autoref{nontight}. In the $2$-frame case we can also use \autoref{lemmoments} to express the bound using moments of mass. We get:
\begin{theorem}\label{thmbuckling}
  Suppose isomtery group $G$ of the domain $\Omega$ admits $2$-frames. For any linear transformation $T$
  \begin{align*}
    \Lambda\Big|_{T(\Omega)}\le \frac{\|T^{-1}\|_2^4+2\|T^{-1}\|_4^4}{(d+2)\|T^{-1}\|_2^2}\Lambda\Big|_\Omega.
  \end{align*}
  Equivalently
  \begin{align*}
    \Lambda V^{2/d}\Big|_{T(\Omega)}\cdot\left.\frac{V^{2/d} I_2}{I_4}\right|_{T^{-1}(\Omega)}
  \end{align*}
  is maximal when $T$ is a multiple of an orthogonal matrix. 
\end{theorem}
Note that similar results can be obtained for any Rayleigh quotient involving two quadratic forms with convex or concave multipliers. In case group $G$ does not allow for appropriate higher order frame, we can use the lower bound from \autoref{nontight} to simplify the left side of \eqref{buckineq}.

\subsection{Numerical comparisons}\label{platenumerical}
\subsubsection{Plate problem}
We wish to compare our upper bound for the lowest plate eigenvalue without tension ($\tau=0$) in dimension $2$ with other known results. The exact eigenvalue for a disk is known, hence we can easily get upper bounds for any ellipse. The disk eigenvalue can be written in terms of the lowest zero of a certain combination of Bessel functions, see e.g. \cite{AL96}. Numerically
\begin{align*}
  \lambda_1(D) \approx \frac{104.36}{r^4},
\end{align*}
for a disk with radius $r$. Therefore Nadirashvili isoperimetric inequality \cite{N95} applied to an ellipse with semiaxes $a$ and $b$ gives
\begin{align*}
  \lambda_1|_{E(a,b)}\ge \frac{104.36}{a^2b^2}.
\end{align*}
Take a linear transformation with singular values $a$ and $b$ to transform a unit disk into an ellipse with semiaxes $a$ and $b$ (ratio $r=a/b$). Then inequality \eqref{platestronger} gives 
\begin{align}\label{ellipseupper}
  \Gamma_1|_{E(a,b)}\le \frac{104.36}{a^2b^2} \frac{3r^2+2+3r^{-2}}{8}\;\left(\le \frac{104.36}{a^2b^2} \frac{r^2+r^{-2}}{2}\right).
\end{align}
The inequality in the parentheses is the weaker upper bound \eqref{plateweaker} (obtained from classical tight frames instead of $2$-frames).

A very comprehensive report by Leissa \cite{L69} contains an extensive section devoted ellipses. Results from various sources, including Shibaoka \cite{S57} and McNitt \cite{McN62}, are compared. Upper bounds (3.8) and (3.6) on page 38 of Leissa \cite{L69} are especially relevant. The bounds are essentially variational with a trial functions given by (3.7) and (3.5). The bounds have the same dependence on semiaxes $a$ and $b$, but differ by a multiplicative constant. This fact should not be a surprise, since both test functions are compositions of linear functions and test functions for disks. As we showed, the Rayleigh quotient for such functions splits into a constant depending on the transformation, and the Rayleigh quotient of the original test function on the disk. Therefore, one might immediately choose the exact eigenfunction for the disk and get the best possible result of this type. In fact Table 3.2 in this report lists upper bounds that seem to be coming from taking the eigenfunction for the disk composed with a linear transformation (instead of using \cite[(3.7)]{L69}), exactly recovering our $2$-frame based result \eqref{platestronger} for the first eigenvalue, without tension.

It is clear, that one can avoid tedious trial function calculations for ellipses due to enough symmetry of the disk. In fact, without knowing the exact value for the disk, we would still get the same semiaxes dependence, and only need to somehow estimate the eigenvalue on the disk, separately. Or simply use the best known upper bound for the disk. It is also worth noting that classical tight $1$-frames are not strong enough to recover the results from Leissa's report. It is crucial to work with $2$-frames, since we get exact values of the Rayleigh quotients.

In \autoref{plateellipse} we compare our bounds with bounds due to McLaurin \cite{McL68}. Note that $2$-frame bound is quite close to his nearly exact values. However, McLaurin is using a collocation method, involving case by case optimization of linear combinations of certain exact solutions. Clearly, this cannot be achieved for arbitrary multipliers.
\begin{table}
  \centering
  \begin{tabular}{lcccc}
  $a/b$                             & $1.1$     & $1.2$     & $2$      & $4$\\
  \hline
  $1$-frames, \eqref{plateweaker}   & $106.262$ & $111.375$ & $221.765$ & $838.1$\\
  $2$-frames, \eqref{platestronger} & $105.786$ & $109.621$ & $192.414$ & $654.7$ \\
  McLaurin (upper)       $\quad$    & $105.741$ & $109.440$ & $187.382$ & $603.2$ \\
  McLaurin (lower)       $\quad$    & $105.741$ & $109.440$ & $187.380$ & $587.2$ \\
  \end{tabular}
  \caption{Fundamental frequency for the tension-less plate problem on ellipse with varying semiaxes ratios $a/b$. Note that Nadirashvili's isoperimetric inequality gives a lower bound $104.36/a^2b^2$ regardless of the ratio $a/b$. All results stated for $ab=1$.}
  \label{plateellipse}
\end{table}

Unfortunately, the stronger $2$-frame result \eqref{platestronger} does not apply to rectangles, since the group of isometries of the square does not admit $2$-frames. This is a fundamental limitation, not related to our method, as we now show. The exact eigenvalue of the square is not known, but Wieners \cite{W96} proved a validated numerical bound $1294.93396\pm0.00002$, while McLaurin \cite{McL68} obtained slightly worse bound using collocation, namely $1294.94\pm0.14$. Taking $1295$ as the eigenvalue of the square, if we falsely applied \eqref{platestronger} to the rectangle with sides $1$ and $2$ we would get the upper bound of around $597$ for that rectangle. Kuttler and Sigillito \cite{KS81} give a lower bound for the same rectangle that equals $603.8$. The same contradiction happens for $4:1$ rectangle, also studied by Kuttler and Sigillito. The upper bound in \eqref{platestronger} is simply wrong for rectangles. One can of course use classical tight frames, but this leads to overestimation of the Rayleigh quotients. One can also check that, the 4th moment of mass $I_4$ for rectangles is underestimated by the $2$-frame averaging.

In view of the rectangles example, it is somewhat surprising that \eqref{platestronger} is actually true for triangles. The lowest plate eigenvalues of the equilateral triangle and the right isosceles triangles have been numerically estimated by Kuttler and Sigillito \cite[Table 4]{KS81}. Rescaling to have area equal $1$, we get $1839$ for the equilateral triangle and $2216$ for the right isosceles triangle. Taking their value for equilateral triangle and finding our stronger upper bound leads to $2389$ as a bound for the right isosceles triangle. Leissa \cite{L69} also lists a few values for isosceles triangles in Table 7.2. After rescaling to area $1$, the equilateral triangle has $1845$, the right isosceles triangle $2190$, and the acute isosceles triangle with angle $30$ degrees has an eigenvalue $2490$. Our upper bound for this triangle is $2910$, while trial function based bound due to Cox and Klein (equation (7.2) in Leissa's report) is about $3400$. The same bound seems better than our result for obtuse isosceles triangles (using Figure 7.2 from Leissa), indicating that their trial function is not just based on a linearly transformed trial function from equilateral triangle. Nevertheless, using our method we get easy to find bounds, that are not too far from the known values. 

Note also that a linear transformation between two triangles is easy to write down as a matrix transforming vertices. The appropriate Schatten norms are then easily calculated without finding singular values.

\subsubsection{Buckling problem} As in the plate problem, we compare our upper bounds with already available results. The exact buckling eigenvalue for a ball is known, while other domains would require a numerical approach. Furthermore, the Faber-Krahn-type isoperimetric inequality is still an open conjecture. Therefore tight bounds are essential. Estimates are available for ellipses, by work of McLaurin \cite{McL69}, giving us some comparisons with our bounds. \autoref{bucklingellipse} summarizes results using $2$-frame based \autoref{thmbuckling}, and a weaker classical frame bounds. The $2$-frame based bound is again quite tight, thanks to exactly evaluated Rayleigh quotient. 

The exact buckling eigenvalue of the unit disk equals $j_{1,1}^2\approx 14.682$ (square of the first zero of Bessel $J_1$ function, see e.g. \cite{AL96}). \autoref{thmbuckling} provides the following bound for an ellipse with semiaxes $a$ and $b$ (ratio $r=a/b$)

\begin{align*}
  \Lambda|_{E(a,b)}\le \frac{j_{1,1}^2}{ab} \frac{3r^4+2r^2+3}{4r(r^2+1)}\;\left(\le \frac{14.682}{ab} \frac{r^4+1}{r(r^2+1)}\right).
\end{align*}
Note that $\frac{j_{1,1}^2}{ab}$ represents the unscaled eigenvalue of a disk with area $ab$. Its numerator corresponds to the scale-invariant product $\Lambda A$ that is conjectured to be minimal for the disk among all domains (P\'olya-Szeg\"o conjecture, see \cite{AL96,ABL97}).

\begin{table}
  \centering
  \begin{tabular}{lccccc}
  $a/b$                             & $1.2$   & $1.4$ & $1.6$ & $2$ & $4$\\
  \hline
  $1$-frames                        & $15.41$ & $17.15$ & $19.47$ & $24.96$ & $55.49$ \\
  $2$-frames, \autoref{thmbuckling} & $15.17$ & $16.34$ & $17.90$ & $21.66$ & $43.34$ \\
  McLaurin (lower bound) $\quad$    & $15.1\;$  & $16.1\;$  & $17.5\;$  & $20.8\;$  & $39.0\;$ \\
  \end{tabular}
  \caption{Fundamental frequency for the buckling problem on ellipse with varying semiaxes ratios $a/b$. Note that exact eigenvalue for a unit disk equals approximately $14.682$. All results stated for $ab=1$.}
  \label{bucklingellipse}
\end{table}

\section{Bochner's subordinators (fractional order operators)}\label{sec:fractional}
In this section we discuss general fractional operators related to Bochner subordination of the Brownian semigroup. This ties our results to the potential theory of the subordinated Brownian motion (see \cite[Chapter 5]{BBKRSV} for a broad overview). At the level of pseudodifferential operators we seek upper bounds for operators $\Psi(-\Delta)$, where the function $\Psi$ is a complete Bernstein function \cite{SSV12}. Weak formulations for such operators have the following frequency domain representation
\begin{align*}
  Q(u,u)=(\Psi(-\Delta)u,u)=\int_{\Rd} \Psi(|\xi|^2) |\widehat u|^2\,d\xi,  
\end{align*}
with the domain $H_0^\beta(\Omega)$ where $\beta\le 1$ depends on $\Psi$. See Chen-Song \cite{CS05,CS06} and Schilling-Song-Vondra{\v{c}}ek \cite{SSV12} for a detailed account.

Even though the above formulation uses complete Bernstein functions, in order to apply our method we only need $\Psi$ to be concave. Therefore our results apply to any weakly defined operator with a multiplier $\Psi(|\xi|^2)$ with concave $\Psi$, as long as its definition makes sense and the spectrum is discrete. 

The quadratic form above fits the framework described in \autoref{qforms}. Therefore we can estimate the eigenvalues of $\Psi(-\Delta)$ on $T(\Omega)$ using the eigenvalues of a slightly modified multiplier on $\Omega$, via \autoref{thmgeneralmult} with $\Phi_1=0$ and $\Phi_2=-\Psi$. 

Suppose $u_i$ are orthonormal eigenfunctions for $\Psi(-\Delta)$ on a domain $\Omega$ with irreducible isometry group. Then \autoref{thmgeneralmult} implies:

\begin{theorem}\label{thmconcave}
  For any concave $\Psi$, linear transformation $T$ and $c=\frac{\|T^{-1}\|_2^2}{d}$ we have
  \begin{align*}
  \lambda_1+\dots+\lambda_n\Big|_{\Psi(|\xi|^2/c),T(\Omega)}&\le
  \lambda_1+\dots+\lambda_n\Big|_{\Psi(|\xi|^2),\Omega}
  \end{align*}
\end{theorem}

Note that unless $\Psi$ is homogeneous we are working with a modified multiplier on $T(\Omega)$, since $\Psi(|\xi|^2/c)$ is not proportional to $\Psi(|\xi|^2)$. We observed the same phenomenon in the bi-Laplacian case with tension term.  Finally, $\|T^{-1}\|_2^2$ is related to the second moment of mass $I_2$ (polar moment of inertia) of $T(\Omega)$ (see \cite[Lemma 9]{LS11} or \autoref{lemmoments} with $p=1$). 

Take $\Psi(t)=t^{\alpha/2}$ ($\alpha/2$-stable subordinator) to get \autoref{thmfractional}. Note that homogeneity of $\Psi$ allows us to restate the result using one operator and using geometric quantities. We can also derive a stronger result for ellipses as images of unit ball $B$ (full orthogonal group as isometry group with integral over Haar measure as averaging). We plug the multiplier into \eqref{averagesum} to get 
\begin{align}
  \lambda_1+\dots+\lambda_n\Big|_{T(B)}\le
  \sum_{i=1}^n\int_{\Rd}\int_{O(d)}|T^{-\dagger}U^\dagger \xi|^\alpha\,d\mu(U)|\widehat u_i|^2d\xi
\end{align}
Rotational invariance of Haar measure allows us to diagonalize $T$ and arrive at \eqref{oversphere} with $2p=\alpha$ (note that $p$ is not an integer here). The eigenvalues of the matrix represent semiaxes of the ellipsoid. For simplicity we only look at ellipses (dimension 2) with semiaxes $a$ and $b$. We get
\begin{align*}
  \lambda_1+\dots+\lambda_n\Big|_{E(a,b)}\le
  \sum_{i=1}^n\int_{\Rd}\left( \frac{1}{2\pi} \int_0^{2\pi} (a^2\cos^2\theta+b^2\sin^2\theta)^\alpha\,d\theta\right) |\xi|^\alpha|\widehat u_i|^2 \,d\xi
\end{align*}
Note that when $\alpha=1$ we recover the perimeter of the ellipse divided by the perimeter of the disk, inside the parentheses. For arbitrary $\alpha$ the resulting integral is a rather complicated combination of hypergeometric functions ${}_2F_1$. At the slight expense of accuracy, for $\alpha\le 1$ we can use Jensen inequality to get perimeter to the power $\alpha$. Now take $u_i$ to be the eigenfunctions on the disk to get \autoref{thmperimeter}.

To obtain this result we essentially used a $1/2$-frame-like identity (we used \eqref{oversphere} with $p=1/2$). This is only possible for ellipses. Indeed, linear transformation of a square could give a rhombus or a rectangle, and the perimeter would not be the same. Similarly for triangles. 

Now we switch to relativistic subordinator $\Psi_m(t)=\sqrt{m^2+t}-m$, which is no longer homogeneous. Therefore the constant $m$ must be rescaled as did $\tau$ in the plate with tension case. Note that
\begin{align*}
  \Psi_m(t/c)=\frac{1}{\sqrt{c}}\left(\sqrt{(\sqrt{c}m)^2+t}-\sqrt{c} m\right).
\end{align*}
Therefore the mass constant needs to scale with $\sqrt{c}$ in order to have equality for eigenvalues of balls with arbitrary radii.  However, $\Psi_m(t)$ as a function of $m$ is decreasing, and for $c>1$
\begin{align*}
  \Psi_m(t/c)\ge \frac{\Psi_m(t)}{\sqrt{c}}.
\end{align*}
Hence \autoref{thmconcave} implies \autoref{thmrelativistic}. 

Finally, when $\Psi$ is any complete Bernstein function, then $\Psi(-\Delta)$ is a generator of a generic subordinated Brownian motion. Concavity of $\Psi$ implies that
\begin{align*}
  f(t/c)\ge \frac{f(t)}{c}.
\end{align*}
Moreover, when $\Psi$ satisfies $\Psi(t/c)\le c^{-\beta}\Psi(t)$ we get \autoref{thmsubor}.

\subsection{Related results}\label{fraccomp}
Even though the spectrum of the fractional Laplacian was studied in many contexts, there are very few known bounds for eigenvalues. Furthermore, it is not even possible to find the exact eigenvalues for intervals, or balls. Kulczycki, Kwa\'snicki, Ma{\l}ecki and St\'os \cite{KKMS10} found very accurate numerical estimates for intervals, while Dyda \cite{D12} found the best (so far) estimates for the first eigenvalue for balls in arbitrary dimension. The only known general bounds are based on isoperimetric inequality and inradius of the domain, by work of Ba\~nuelos, Lata{\l}a and M{\'e}ndez-Hern{\'a}ndez \cite{BLM01}. In this case, the eigenvalues of a general domain are estimated by the eigenvalue of a ball. 

One can also find bounds for the eigenvalues of the fractional Laplacian by relating them to the eigenvalues of the Laplacian, using quadratic form comparability (Chen-Song \cite{CS05}). 

  We propose another approach, based on comparisons to ellipsoids, instead of balls. This gives sharper results for elongated domains, comparing to inradius based approach. Let $E\subset \Omega$ be a John ellipsoid of $\Omega$ (ellipsoid with maximal volume contained in $\Omega$, see John \cite{J48} or Ball \cite{B92}). Then $d E$ ($\sqrt{d} E$ if $\Omega$ is centrally symmetric) contains $\Omega$. Hence
\begin{align*}
  (\lambda_1+\dots+\lambda_n)c^{-\alpha/2}\Big|_{\alpha,E}\le 
\lambda_1+\dots+\lambda_n\Big|_{\alpha,\Omega}\le
\lambda_1+\dots+\lambda_n\Big|_{\alpha,E},
\end{align*}
where $c$ equals $d$ for general domains, or $\sqrt{d}$ for centrally symmetric domains. It is also possible to find an optimal constant $c$, for a specific domain $\Omega$, which will be no worse than the general case. Now we find $T$ such that $T^{-1}(E)$ is a unit disk, and use our bounds to get an upper bound on $\Omega$.

Let $\Omega$ be a plane domain with John ellipse $E$ with semiaxes of length $1$ and $a>1$. There exists a linear transformation $T^{-1}$ that takes this ellipse to a unit disk $D$. It also transforms the original domain into $\Omega'$ with John ellipse $D$ (the ellipse is a disk). \autoref{thmfractional} (or \autoref{thmsubor}) gives the following bound for the lowest eigenvalue of the original domain
\begin{align*}
  \lambda_1 \Big|_{\alpha,\Omega}\le \lambda_1\Big|_{\alpha,E} \le \frac{\|T^{-1}\|_{HS}^\alpha}{2^{\alpha/2}}\lambda_1\Big|_{\alpha,D}= \left(\frac{1+1/a^2}{2}\right)^{\alpha/2}\lambda_1\Big|_{\alpha,D}
\end{align*}
The domain $\Omega$ has inradius $r$, that satisfies $1\le r^2\le a$, since John ellipse has the largest volume of all ellipses inside $\Omega$, and inradius is at least as large as the short semiaxis in John ellipse. Inradius bound from \cite{BLM01} reads
\begin{align*}
  \lambda_1\Big|_{\alpha,\Omega}\le \frac{1}{r^\alpha} \lambda_1\Big|_{\alpha,D}.
\end{align*}
It is not clear if our bound is better in general then inradius based bound due to Ba\~nuelos et. al. \cite{BLM01}. If $\Omega$ is a rectangle (or ellipse), then $r=1$ and our bound is certainly better. Applying the same linear transformation $T$ to rotated squares leads to parallelograms. The maximal inradius is achieved when the image is a rhombus. Surprisingly, one checks that in this case both bounds give the same result. Hence our bound is at least as good on parallelograms as the inradius based bound. On the other hand, the discussion in \cite[Section 8.1]{LS11plane} shows that inradius based bound is better than our bound for all triangles. In general, inradius based bound should be better for polygons with inscribed circle and other nearly round domains. While our bound should be stronger on elongated domains. 

For $\alpha\le 1$ we can also use \autoref{thmperimeter} to get stronger upper bounds for $\lambda_1|_{\alpha,E}$, however this involves the perimeter of ellipse $E$ (elliptic integral).

\section*{Acknowledgements}
The work was partially supported by NCN grant 2012/07/B/ST1/03356.

The author is grateful to Richard Laugesen for invaluable discussions on the spectral theory and the plate problems, as well as suggested improvements to some arguments. The author would also like to thank Jonathan Brundan for his guidance on the algebraic topics.

\end{document}